\documentclass[a4paper,11pt,reqno,noindent]{amsart}
\usepackage[centertags]{amsmath}
\usepackage{amsfonts,amssymb,amsthm,dsfont,cases,amscd,esint,enumerate}
\usepackage[T1]{fontenc}
\usepackage[english]{babel}
\usepackage[applemac]{inputenc}
\usepackage{newlfont}
\usepackage{color}
\usepackage[body={15cm,21.5cm},centering]{geometry} 
\usepackage{fancyhdr}
\usepackage{enumerate}

\usepackage[colorlinks,citecolor=black,urlcolor=black]{hyperref}

\theoremstyle{plain}
\newtheorem{theor0}{Theorem}[section]
\newenvironment{theor}
  {\pushQED{\qed}\begin{theor0}}{\popQED\end{theor0}}
\newtheorem{lem0}[theor0]{Lemma}
\newenvironment{lem}
  {\pushQED{\qed}\begin{lem0}}{\popQED\end{lem0}}
\newtheorem{prop0}[theor0]{Proposition}
\newenvironment{prop}
  {\pushQED{\qed}\begin{prop0}}{\popQED\end{prop0}}
\newtheorem{cor0}[theor0]{Corollary}

\newtheorem{propr0}[theor0]{Property}

\newtheorem{hyp0}[theor0]{Hypothesis}

\newtheorem{result0}[theor0]{Result}

\newtheorem{conj0}[theor0]{Conjecture}

\newtheorem{heur0}[theor0]{Heuristics}

\theoremstyle{definition}
\newtheorem{defin0}[theor0]{Definition}
\newenvironment{defin}
  {\pushQED{\qed}\begin{defin0}}{\popQED\end{defin0}}
\newtheorem{rems0}[theor0]{Remarks}

\newtheorem{ex0}[theor0]{Example}

\newtheorem{exs0}[theor0]{Examples}

\newtheorem{rem0}[theor0]{Remark}
\newenvironment{rem}
  {\pushQED{\qed}\begin{rem0}}{\popQED\end{rem0}}
\newtheorem{qu0}[theor0]{Question}

\newtheorem{qus0}[theor0]{Questions}

  \newtheorem{as0}[theor0]{Assumption}

\newcommand{\Aa}{A}
\newcommand{\bb}{B}

\newcommand{\N}{\mathbb N}
\newcommand{\e}{\varepsilon}

\newcommand{\R}{\mathbb R}
\newcommand{\Z}{\mathbb Z}

\newcommand{\B}{\mathcal B}
\newcommand{\X}{\mathcal X}

\newcommand{\Cc}{\mathcal C}
\newcommand{\Sp}{\mathbb S}
\newcommand{\Rc}{\mathcal R}

\newcommand{\Pc}{\mathcal P}

\newcommand{\Dm}{\mathbb{D}}

\newcommand{\F}{\mathcal F}
\newcommand{\Hf}{\mathfrak H}
\newcommand{\Nc}{\mathcal N}

\newcommand{\A}{\mathcal A}
\newcommand{\p}{\mathbb{P}}
\newcommand{\E}{\mathbb{E}}
\newcommand{\calS}{\mathcal S}
\newcommand{\op}{{\operatorname{op}}}

\newcommand{\osc}{\partial^{\mathrm{osc}}}
\newcommand{\oscd}[1]{\mathrel{\osc_{#1}}}
\newcommand{\supess}{\operatorname{sup\,ess}}
\newcommand{\supessd}[1]{\mathrel{\mathop{\supess}\limits_{#1}}}
\newcommand{\infessd}[1]{\mathrel{\mathop{\infess}\limits_{#1}}}
\newcommand{\infess}{\operatorname{inf\,ess}}
\newcommand{\Mes}{\operatorname{Mes}}
\newcommand{\loc}{{\operatorname{loc}}}
\newcommand{\Ld}{\operatorname{L}}
\newcommand{\supp}{\operatorname{supp}}

\newcommand{\3}{\operatorname{|\!|\!|}}
\newcommand{\var}[1]{\mathrm{Var}\left[#1\right]}
\newcommand{\ent}[1]{\mathrm{Ent}\!\left[#1\right]}

\newcommand{\varm}[1]{\mathrm{Var}[#1]}
\newcommand{\expecm}[1]{\mathbb{E}[ #1 ]}
\newcommand{\cov}[2]{\mathrm{Cov}\left[#1;#2\right]}
\newcommand{\expec}[1]{\mathbb{E}\left[ #1 \right]}
\newcommand{\pr}[1]{\mathbb{P}\left[ #1 \right]}
\newcommand{\prm}[1]{\mathbb{P}\big[ #1 \big]}

\newcommand{\expeCm}[2]{\mathbb{E}\big[ #1 \,\big\|\,#2\big]}

\newcommand{\parfct}[1]{\partial_{#1}^{\operatorname{fct}}}

\newcommand{\pardis}[1]{\partial_{#1}^{\operatorname{dis}}}

\newcommand{\dTV}[2]{\operatorname{d}_{\operatorname{TV}}\left({#1},{#2}\right)}
\newcommand{\dW}{\operatorname{d}_{\operatorname{W}}}
\newcommand{\dK}{\operatorname{d}_{\operatorname{K}}}
\newcommand{\step}[1]{\noindent \textit{Step} #1.}

\numberwithin{equation}{section}

\title[Multiscale second-order Poincar\'e inequalities]{Multiscale second-order Poincar\'e inequalities\\in probability}

\author[M. Duerinckx]{Mitia Duerinckx}
\author[A. Gloria]{Antoine Gloria}
\address[Mitia Duerinckx]{Laboratoire de Mathématique d'Orsay, UMR 8628, Université Paris-Sud, F-91405 Orsay, France \& Universit\'e Libre de Bruxelles, Département de Mathématique, Brussels, Belgium}
\email{mduerinc@ulb.ac.be}
\address[Antoine Gloria]{Sorbonne Universit\'e, CNRS, Universit\'e de Paris, Laboratoire Jacques-Louis Lions (LJLL), F-75005 Paris, France \& Universit\'e Libre de Bruxelles, Département de Mathématique, Brussels, Belgium}
\email{gloria@ljll.math.upmc.fr}

\begin{document}
\maketitle

\begin{abstract}
Consider an ergodic stationary random field $A$ on the ambient space $\R^d$.
In a companion article, we introduced the notion of \emph{multiscale} (first-order) \emph{functional inequalities}, which extend
standard functional inequalities like Poincaré, covariance, and logarithmic Sobolev inequalities in the probability space, while still ensuring strong concentration properties.
We also developed a constructive approach to these functional inequalities, proving their validity
for prototypical examples including Gaussian fields, Poisson random tessellations, and random sequential adsorption (RSA) models, which do not satisfy standard functional inequalities.
In the present contribution, we turn to second-order Poincar\'e inequalities à la Chatterjee: while first-order inequalities quantify the distance to constants for nonlinear functions $Z(A)$ in terms of their local dependence on the random field $A$, second-order inequalities quantify their distance to normality.
For the above-mentioned examples, we prove the validity of suitable \emph{multiscale} second-order Poincar\'e inequalities. In particular, applied to RSA models, these functional inequalities allow to complete and improve previous results by Schreiber, Penrose, and Yukich on the jamming limit, and to propose and fully analyze a more efficient algorithm to approximate the latter.
\end{abstract}

\tableofcontents

\section{Introduction}

{\color{black}
Consider a random field $A$ on $\R^d$ and a $\sigma(A)$-measurable random variable $Z(A)$.
Chatterjee's version of Stein's method in the form of second-order Poincar\'e inequalities~\cite{Chat08,Chat-09,NPR-09} is a powerful tool to quantify the distance of $Z(A)$ to normality in terms of its local dependence on the underlying random field $A$ (up to second variation) via suitable ``derivatives''.
When applicable, this often leads to strong quantitative central limit theorems (CLTs).
However, similarly as first-order functional inequalities in the probability space (Poincaré, covariance, and logarithmic Sobolev inequalities), second-order inequalities essentially hold true only when $A$ has a product structure (e.g.\@ a Poisson point process) or is Gaussian with integrable covariance.
In this contribution, motivated by problems stemming from the analysis of heterogeneous materials, we aim
at extending such results by devising suitable versions of second-order Poincar\'e inequalities to treat various random fields~$A$ with strong correlations, {\color{black}including various examples used for modelling in materials science~\cite{Torquato-02}.}
This is paralleled by our companion works~\cite{DG2,DG1} on first-order functional inequalities.

\medskip

We specifically consider random fields $A$ on $\R^d$ that can be constructed as deformations $A=\Phi(A_0)$ of (typically higher-dimensional) random fields $A_0$ that are known to satisfy a standard second-order Poincar\'e inequality.
For a $\sigma(\Aa)$-measurable random variable $Z(\Aa)$,
we can write $Z(A)=X\circ \Phi(A_0)$
and appeal to the second-order Poincar\'e inequality wrt $A_0$ to estimate the distance of $Z(A)$ to normality.
However, even though the map $A\mapsto Z(A)$ is well-controlled, the map $A_0\mapsto X\circ\Phi(A_0)$ might be
much more intricate if $\Phi$ is a complicated object (e.g.~the graphical construction to pass from the Poisson point process on $\R^d\times \R_+$ to the random parking measure on $\R^d$, cf.~\cite{Penrose-01}).
We aim at devising a proxy for a chain-rule and deriving a suitable second-order Poincar\'e inequality wrt $A$ from its known counterpart wrt $A_0$.
Following the strategy of the companion article~\cite{DG2}, we are then led to multiscale weighted versions of second-order Poincar\'e inequalities, where the local dependence wrt $A$ is considered on all scales and suitably weighted. The strategy and the main general results
are presented in Section~\ref{sec:alaChatt}.

\medskip

Before we turn to applications, we briefly comment on the arguments of Section~\ref{sec:alaChatt} to establish multiscale second-order functional inequalities.
As in~\cite{DG2}, two prototypical classes of examples are distinguished:
\begin{itemize}
\item {\it Deterministic localization}, that is, when the random field $A=\Phi(A_0)$ is a deterministic convolution $\Phi$ of a product structure $A_0$, so that the ``dependence pattern'' is prescribed deterministically a priori. This mainly concerns Gaussian fields, which can indeed be seen as convolutions of a white noise and which have been extensively studied in the literature in the framework of Malliavin calculus.
The multiscale second-order Poincar\'e inequalities that are obtained here constitute an alternative formulation of such Malliavin results~\cite{NPR-09,Nourdin-Peccati-12,Vidotto}
and involve slightly more elementary objects (Fr\'echet derivatives wrt $A$ instead of Malliavin derivatives), thus making it easier to apply.
\smallskip\item
{\it Random localization}, that is, when the ``dependence pattern'' of $A=\Phi(A_0)$ is encoded in the underlying product structure $A_0$ itself, hence varies with the realization.
More precisely, we shall discuss a few typical examples that can be viewed as transformations of a Poisson point process: random Poisson inclusions with i.i.d.\@ random radii, random Poisson tessellations, and the celebrated random parking process.
In order to control and quantify the nonlocality of the transformation~$\Phi$, we exploit its stabilization properties, drawing inspiration from the works by Penrose and Yukich~\cite{Penrose-01,Penrose-Yukich-02,PY-05}. As in the companion article~\cite{DG2}, stabilization is efficiently formulated in terms of an ``action radius''.
Our results naturally compare to~\cite{LPS-16}, where a general strategy was independently developed to prove approximate normality results for functionals of Poisson processes based on stabilization: in the present contribution, stabilization properties are confined to the proof of multiscale second-order Poincar\'e inequalities and the focus is rather on these functional inequalities in their own right --- which can then be applied without further mention of stabilization issues.
\end{itemize}

\medskip
In Section~\ref{chap:appl}, we use multiscale second-order Poincaré inequalities to establish quantitative CLTs for (linear) spatial averages of the random field~$A$. Although these functional inequalities are designed to address general {\it nonlinear} functions of~$A$, even this application to \emph{linear} random variables is nontrivial and is particularly relevant in two contexts: the analysis of fluctuations in stochastic homogenization and in stochastic geometry.
On the one hand, in the context of quantitative stochastic homogenization for random linear elliptic operators in divergence form (that is, operators of the form $-\nabla\cdot A\nabla$ with a random coefficient field $A$), various quantities of interest are proved to behave essentially like spatial averages of the random field, and applying second-order Poincar\'e inequalities then leads to sharp normal approximation results~\cite{Nolen-13,GN,GuM,MN,DGO1,DGO2,DFG}.
This contribution provides functional-analytic
tools to extend quantitative homogenization results to this setting.
On the other hand, we believe that the ideas developed in this contribution could be useful to revisit and extend various quantitative CLT results in stochastic geometry (e.g.~\cite{Penrose-05,PY-05,HLS-16}), possibly improving on convergence rates and also allowing to replace underlying Poisson point processes by more correlated processes.
As a first illustration, we establish a quantitative CLT with optimal rate for spatial averages of a Poisson inclusion model with random radii, which improves on previous results for Boolean models~\cite{Heinrich-05,HLS-16}.
As a second example, we turn to the jamming limit for RSA models (cf.~\cite{Penrose-Yukich-02}, and the subsequent works~\cite{Penrose-01,Penrose-Yukich-02,Penrose-05,PY-05,Schreiber-Penrose-Yukich-07,LPS-16}), which we revisit and complete  by using multiscale first- and second-order functional inequalities.

\medskip

In principle, all the results that can be proved using multiscale second-order Poincar\'e inequalities could be proved using the approaches of~\cite{NPR-09} or~\cite{LPS-16}. As made clear in this introduction, the merit of this contribution is  twofold.
First, it shows that the multiscale weighted form of first-order functional inequalities introduced and studied in~\cite{DG2,DG1} is also meaningful in the context of second-order inequalities.
Second, it provides an original intrinsic formulation of second-order Poincar\'e inequalities for correlated fields $A=\Phi(A_0)$ that is oblivious of the ``hidden'' product structure $A_0$ 
and is only formulated in terms of ``derivatives'' wrt $A$. As a consequence, when applicable, such inequalities allow to estimate distance to normality in a more straightforward and transparent way than other methods.}

\bigskip

{\bf Notation.} 
\begin{itemize}
\item $d$ is the dimension of the ambient space $\R^d$;
\item $C$ denotes various positive constants that only depend on the dimension $d$ and possibly on other controlled quantities; we write $\lesssim$ and $\gtrsim$ for $\le$ and $\ge$ up to such multiplicative constants $C$; we use the notation $\simeq$ if both relations $\lesssim$ and $\gtrsim$ hold; we add a subscript in order to indicate the dependence of the multiplicative constants on other parameters;
\item $Q^k:=[-1/2,1/2)^k$ denotes the unit cube centered at $0$ in dimension $k$,
and for all  $x\in\R^k$ and $r>0$ we set $Q^k(x):=x+Q^k$, $Q^k_r:=rQ^k$ and $Q^k_r(x):=x+rQ^k$; when $k=d$ or when there is no confusion possible on the meant dimension, we drop the superscript $k$;
\item we use similar notation for balls, replacing $Q^k$ by $B^k$ (the unit ball in dimension~$k$);
\item the Euclidean distance between  subsets of $\R^d$ is denoted by $d(\cdot,\cdot)$;
\item $\B(\R^k)$ denotes the Borel $\sigma$-algebra on $\R^k$;
\item for all Borel sets $S\subset \R^d$ with finite measure, $\fint_S$ denotes the averaged Lebesgue integral on $S$;
\item $\expec{\cdot}$ denotes the expectation, $\var{\cdot}$ the variance, and $\cov{\cdot}{\cdot}$ the covariance in the underlying probability space $(\Omega,\A,\p)$,
and the notation $\expec{\cdot \| \cdot}$ stands for the conditional expectation;
\item for all $a,b\in \R$, we set $a\wedge b:=\min\{a,b\}$, $a\vee b:=\max\{a,b\}$, and $a_+:=a\vee0$;
\item $\mathcal N$ denotes a standard normal random variable;
\item $\dTV{\cdot}{\cdot}$, $\dW(\cdot,\cdot)$, and  $\dK(\cdot,\cdot)$ denote the total variation, the $1$-Wasserstein, and the Kolmogorov distances, respectively.
\end{itemize}

\section{Multiscale second-order Poincar\'e inequalities}\label{sec:alaChatt}

Chatterjee's standard second-order Poincar\'e inequalities are known to hold
for product structures in Wasserstein and Kolmogorov distances~\cite{Chat08,Lachieze-Peccati-16}.
Based on these results, we prove the validity of multiscale versions of such functional inequalities 
for random fields that are deformations of product structures.
To this aim, we first recall the constructive approach of the companion article~\cite{DG2} for multiscale first-order functional inequalities and then turn to the two prototypical classes of deformations: deterministic localization (which essentially concerns Gaussian fields) and random localization (in which case the nonlocality is controlled in terms of an action radius, cf.~\cite{DG2}).

\subsection{Multiscale first-order functional inequalities}

Let $A:\R^d\times\Omega\to\R$ be a jointly measurable random field on $\R^d$, constructed on a probability space $(\Omega,\A,\p)$.
In this continuum setting, we start by recalling two notions of ``derivatives'' wrt the random field~$A$, which measure the sensitivity upon local restrictions of $A$.
\begin{itemize}
\item The {\it oscillation} $\osc$ is formally defined for any Borel set $S\subset \R^d$ by
\begin{multline}
\hspace{1cm}\oscd{A,S}Z(A)~:=~\supessd{A,S}Z(A)-\infessd{A,S}Z(A)\\
=~ \supess\left\{Z(A'):A'\in\Mes(\R^d;\R),\,A'|_{\R^d\setminus S}=A|_{\R^d\setminus S}\right\} \\
-\infess\left\{Z(A'):A'\in\Mes(\R^d;\R),\,A'|_{\R^d\setminus S}=A|_{\R^d\setminus S}\right\},\label{e.def-Glauber-formal}
\end{multline}
where the essential supremum and infimum are taken wrt the measure induced by the field $A$ on the space $\Mes(\R^d;\R)$ (endowed with the cylindrical $\sigma$-algebra). We refer to~\cite{DG2} for more careful definitions.
\smallskip\item The (integrated) {\it functional derivative} $\parfct{}$ is defined as follows.
Choose an open set $M\subset \Ld^\infty(\R^d)$ containing the realizations of the random field $A$. Given a $\sigma(A)$-measurable random variable $Z(A)$ and given an extension $\tilde Z:M\to\R$, its Fr\'echet derivative $\frac{\partial\tilde Z(A)}{\partial A}\in\Ld^1_\loc(\R^d)$ is defined for all compactly supported perturbation~$B\in\Ld^\infty(\R^d)$ by
\begin{align}\label{eq:def-frech}
\qquad\lim_{t\to0}\frac{\tilde Z(A+tB)-\tilde Z(A)}t=\int_{\R^d}B(x)\frac{\partial\tilde Z(A)}{\partial A}(x)\,dx,
\end{align}
if the limit exists. {\color{black}(The extension $\tilde Z$ is only needed to make sure that quantities like $\tilde Z(A+tB)$ make sense for small $t$, while $Z$ is a priori only defined on realizations of $A$. In the sequel we will always assume that such an extension is implicitly given; this is typically the case  in applications in stochastic homogenization.)}
Since we are interested in the local averages of this derivative, we rather define for all bounded Borel subsets $S\subset\R^d$,
\begin{align*}
\qquad&\partial_{A,S}^{\operatorname{fct}}Z(A)=\int_S\Big|\frac{\partial\tilde Z(A)}{\partial A}(x)\Big|dx,
\end{align*}
which is alternatively characterized by
\begin{align*}
\hspace{1.2cm}\parfct{A,S} Z(A)
\,=\,\sup\Big\{ \limsup_{t\downarrow 0} \frac{\tilde Z(A+t B)-\tilde Z(A)}{t} \,:\, \supp{B}\subset S,\,\sup |B|\le 1\Big\}.
\end{align*}
This derivative is additive wrt the set $S$: for all disjoint Borel subsets $S_1,S_2\subset\R^d$ we have $\partial_{A,S_1\cup S_2}^{\operatorname{fct}}Z(A)=\partial_{A,S_1}^{\operatorname{fct}}Z(A)+\partial_{A,S_2}^{\operatorname{fct}}Z(A)$.
\end{itemize}

Henceforth we use $\tilde\partial$ to denote either $\parfct{}$ or $\osc$. We now recall the form of the multiscale Poincaré inequality introduced in~\cite{DG2}.

\begin{defin}\label{defi:WFI}
Given an integrable function $\pi:\R_+\to\R_+$, we say that $A$ satisfies the {\it multiscale Poincaré} (or spectral gap) {\it inequality \emph{($\tilde \partial$-WSG)} with weight $\pi$} if for all $\sigma(A)$-measurable random variables $Z(A)$ we have
\begin{equation}\label{eq:wsg1}
\var{Z(A)}\le \,\expec{\int_0^\infty \int_{\R^d}\Big(\tilde\partial_{A,B_{\ell+1}(x)}Z(A)\Big)^{2}dx\,(\ell+1)^{-d}\pi(\ell)\,d\ell}.\qedhere
\end{equation}
\end{defin}

This functional inequality reduces to the Poincaré inequality when the weight $\pi$ is compactly supported.
When $\pi$ is not compactly supported, the above takes into account the dependence of $Z(A)$ upon local restrictions of $A$ on arbitrarily large sets (with a suitably decaying weight), thus encoding the dependence pattern of $A$.
For instance, in the specific example of a random Poisson tessellation on $\R^d$, the multiscale inequality~\eqref{eq:wsg1} has been shown to hold
with $\tilde \partial=\osc$ and with weight $\pi(\ell)=Ce^{-\frac1C\ell^d}$ (cf.~\cite{DG2}), where the weight $\pi(\ell)$ corresponds to the probability for an
element of the tessellation to have diameter $>\ell$.

In the companion article~\cite{DG2}, we have developed a general constructive approach to such functional inequalities in the case when the random field $A$ is a deformation of a product structure.
More precisely, let the random field $A$ on $\R^d$ be $\sigma(\X)$-measurable for some random field $\X$ defined on some measure space $E$ and with values in some measurable space $M$. Assume that we have a partition $E=\biguplus_{x\in\Z^d,t\in\Z^l}E_{x,t}$ on which $\X$ is {\it completely independent}, that is, the restrictions $(\X|_{E_{x,t}})_{x\in\Z^d,t\in\Z^l}$ are all independent.
The random field $\X$ can be e.g.\@ a random field on $\R^d\times\R^l$ with values in some measure space (choosing $E=\R^d\times\R^l$, $E_{x,t}=Q^d(x)\times Q^l(t)$, and $M$ the space of values), or a random point process or a random measure on $\R^d\times\R^l\times E'$ for some measure space $E'$ (choosing $E=\Z^d\times\Z^l\times E'$, $E_{x,t}=\{x\}\times\{t\}\times E'$, and $M$ the space of measures on $Q^d\times Q^l\times E'$).

Let $\X'$ be some given i.i.d.\@ copy of $\X$. For all $x,t$, we define a perturbed random field $\X^{x,t}$ by setting $\X^{x,t}|_{E\setminus E_{x,t}}=\X|_{E\setminus E_{x,t}}$ and $\X^{x,t}|_{E_{x,t}}=\X'|_{E_{x,t}}$. By complete independence, the random fields $\X$ and $\X^{x,t}$ (resp.\@ $A=A(\X)$ and $A(\X^{x,t})$)  have the same law.
The following standard discrete version of the Poincaré inequality is known as the Efron-Stein inequality~\cite{Efron-Stein-81,Steele-86} (see also~\cite[Proposition~2.4]{DG2}).
\begin{prop}\label{prop:sgindep}
For all $\sigma(\X)$-measurable random variables $Z(\X)$, we have
\begin{gather*}
\var{Z(\X)}\le\frac12\,\sum_{x\in\Z^d}\sum_{t\in\Z^l}\expec{\big(Z(\X)-Z(\X^{x,t})\big)^2}.
\qedhere
\end{gather*}
\end{prop}

In~\cite{DG2}, we describe general situations for which the above standard Poincaré inequality for the ``hidden'' product structure $\X$ is deformed into a multiscale functional inequality of the form~\eqref{eq:wsg1} for the random field $A$.
 As recalled in the introduction, two main situations need to be distinguished for the deformation: deterministic localization (which mainly concerns Gaussian fields) and random localization (in which case the nonlocality of the deformation depends on the realization).
The same distinction is needed when we turn to second-order Poincaré inequalities.
In Section~\ref{sec:ran-loc}
we start with the more original case of random localization and then quickly
address the case of Gaussian fields in Section~\ref{sec:det-loc}.

\subsection{Random localization}\label{sec:ran-loc}

Let $A$ be a $\sigma(\X)$-measurable random field on $\R^d$, where~$\X$ is completely independent on some measure space $E=\biguplus_{x\in\Z^d,t\in\Z^l}E_{x,t}$ with values in some measurable space $M$.
We address the case of random localization.
In this context, we first recall the crucial notion of {\it action radius} (cf.~\cite{DG2}), which is a probabilistic measure of the nonlocality of the dependence pattern. It is inspired by the stabilization radius first introduced by Lee~\cite{Lee-97,Lee-99} and crucially used in the works by Penrose, Schreiber, and Yukich on RSA processes~\cite{Penrose-Yukich-02,Penrose-05,PY-05,Schreiber-Penrose-Yukich-07}.
\begin{defin}\label{def:ar}
Given an i.i.d.\@ copy $\X'$ of the field $\X$, an {\it action radius for $A$ wrt $\X$ on $E_{x,t}$} (with reference perturbation $\X'$), if it exists, is defined as a nonnegative $\sigma(\X,\X')$-measurable random variable $\rho$ such that we have a.s.,
\[\left.A(\X^{x,t})\right|_{\R^d\setminus (Q(x)+B_\rho)}=\left.A(\X)\right|_{\R^d\setminus(Q(x)+B_\rho)},\]
where we recall that the perturbed random field $\X^{x,t}$ is defined by $\X^{x,t}|_{E\setminus E_{x,t}}:=\X|_{E\setminus E_{x,t}}$ and $\X^{x,t}|_{E_{x,t}}:=\X'|_{E_{x,t}}$.
\end{defin}
The following result establishes a multiscale second-order Poincar\'e inequality for $A$, based on assumptions on a slightly stronger notion of action radius. 
In the statement we define and use discrete derivatives, which are directly controlled by the oscillation~\eqref{e.def-Glauber-formal}.
The strategy consists in applying the standard second-order Poincar\'e inequality for $\X$ due to Chatterjee~\cite{Chat08}, and then exploiting the localization properties of the action radius to devise an approximate chain rule and deduce a functional inequality wrt $A=A(\X)$ itself. As already discussed, this is to be compared with~\cite{LPS-16}.
\begingroup\allowdisplaybreaks
\begin{theor}\label{th:chat-ra}
Let $A$ be a $\sigma(\X)$-measurable random field on $\R^d$, where $\X$ is a completely independent random field on some measure space $E=\biguplus_{x\in\Z^d,t\in\Z^l}E_{x,t}$ with values in some measurable space $M$.
Let $\X'$ be an i.i.d.\@ copy of $\X$. For all $B\subset \Z^d\times\Z^l$, let the perturbed random field $\X^{B}$ be defined by
\begin{gather*}
\X^{B}|_{\cup_{(x,t)\in B}E_{x,t}}=\X'|_{\cup_{(x,t)\in B}E_{x,t}},\qquad\X^{B}|_{\cup_{(x,t)\notin B}E_{x,t}}=\X|_{\cup_{(x,t)\notin B}E_{x,t}},
\end{gather*}
and for all $x,x'\in \Z^d$ and $t,t'\in\Z^l$ we set for simplicity $\X^{x,t}:=\X^{\{(x,t)\}}$ and $\X^{x,t;x',t'}:=\X^{\{(x,t),(x',t')\}}$. Assume that
\begin{enumerate}[(a)]
\item For all $x,t$ and all $B\subset\Z^d\times\Z^l$, there exists an action radius $\rho_{x,t}(\X^B)$ for $A(\X^B)$ wrt $\X^B$ in $E_{x,t}$ with reference perturbation $\X'$ (in the sense of Definition~\ref{def:ar}), and set
\[\tilde\rho_{x,t}:=\sup\big\{\rho_{x,t}(\X^B)\,:\,B\subset\Z^d\times\Z^l\big\}.\]
\item The transformation $A$ of $\X$ is stationary, that is, the random fields $A(\X(\cdot+z,\cdot))$ and $A(\X)(\cdot+z)$ have the same law for all $z\in\Z^d$.
\end{enumerate}
For all $t\in\Z^l$ and $\ell\ge1$, define the weight
\[\pi(t,\ell):=\pr{\ell-1\le\tilde\rho_{0,t}<\ell\,,~\X\ne\X^{0,t}}.\]
Then the following results hold.
\begin{enumerate}[(i)]
\item For all $\sigma(A)$-measurable random variables~$Z=Z(A)$ with $0<\sigma^2:=\var{Z}<\infty$, we have
\begin{eqnarray}\label{eq:chat-ar-gen}
\qquad\lefteqn{\dW\big(\sigma^{-1}({Z-\expec{Z}}),\Nc\big)}\nonumber\\
& \lesssim &\sigma^{-2}\inf_{0<\lambda<1} \Bigg( \sum_{x,x',x''}\sum_{t,t',t''}\sum_{\ell,\ell',\ell''=1}^\infty \big(\pi(t,\ell)^\frac13\pi(t',\ell')^\frac13\pi(t'',\ell'')^\frac13\big)^{\lambda}\nonumber\\
&&\hspace{3cm}\times\expec{ \Big(\pardis{\ell,x,t}\pardis{\ell',x',t'}Z\Big)^\frac4{1-\lambda}}^\frac{1-\lambda}4\expec{ \Big(\pardis{\ell,x,t}\pardis{\ell'',x'',t''}Z\Big)^\frac4{1-\lambda}}^\frac{1-\lambda}4\nonumber\\
&&\hspace{4cm}\times\expec{\Big(\pardis{\ell',x',t'} Z\Big)^\frac4{1-\lambda}}^\frac{1-\lambda}4\expec{\Big(\pardis{\ell'',x'',t''} Z\Big)^\frac4{1-\lambda}}^\frac{1-\lambda}4\Bigg)^\frac12\nonumber\\
&&+\sigma^{-2}\inf_{0<\lambda<1} \Bigg( \sum_{x,x'}\sum_{t,t'}\sum_{\ell,\ell'=1}^\infty \big(\pi(t,\ell)^\frac12\pi(t',\ell')^\frac12\big)^{\lambda}\nonumber\\
&&\hspace{3cm}\times\expec{ \Big(\pardis{\ell,x,t}\pardis{\ell',x',t'}Z\Big)^\frac4{1-\lambda}}^\frac{1-\lambda}2 \expec{\Big(\pardis{\ell',x',t'} Z\Big)^\frac4{1-\lambda}}^\frac{1-\lambda}2\Bigg)^\frac12\nonumber\\
&&+\sigma^{-2}\inf_{0<\lambda<1} \Bigg( \sum_{x}\sum_t\sum_{\ell=1}^\infty \pi(t,\ell)^{\lambda}\expec{ (\pardis{\ell,x,t}Z)^\frac4{1-\lambda}}^{1-\lambda}\Bigg)^\frac12\nonumber\\
&&+\sigma^{-3} \inf_{0<\lambda<1} \sum_{x}\sum_t\sum_{\ell=1}^\infty \pi(t,\ell)^{\lambda} \expec{\Big(\pardis{\ell,x,t}Z\Big)^{\frac3{1-\lambda}}}^{1-\lambda},
\end{eqnarray}
where the sums in $x,x',x''$ (resp.\@ in $t,t',t''$) implicitly run over $\Z^d$ (resp.\@ over $\Z^l$), and where for all $x\in\Z^d$ and $t\in\Z^l$ we have defined the discrete derivative
\[\pardis{\ell,x,t}Z:=\big(Z(A)-Z(A(\X^{x,t}))\big)\mathds1_{A(\X^{x,t})|_{\R^d\setminus Q_{2\ell+1}(x)}=A|_{\R^d\setminus Q_{2\ell+1}(x)}}\]
and the discrete second derivative
\begin{multline*}
\qquad\pardis{\ell,x,t} \pardis{\ell',x',t'} X:=\big(Z(A)-Z(A(\X^{x,t}))-Z(A(\X^{x',t'}))+Z(A(\X^{x,t;x',t'}))\big)\\
\times \mathds1_{A(\X^{x,t})|_{\R^d\setminus Q_{2\ell+1}(x)}=A|_{\R^d\setminus Q_{2\ell+1}(x)}} \mathds1_{A(\X^{x,t;x',t'})|_{\R^d\setminus Q_{2\ell+1}(x)}=A(\X^{x',t'})|_{\R^d\setminus Q_{2\ell+1}(x)}}\\
\times \mathds1_{A(\X^{x',t'})|_{\R^d\setminus Q_{2\ell'+1}(x')}=A|_{\R^d\setminus Q_{2\ell'+1}(x')}} \mathds1_{A(\X^{x,t;x',t'})|_{\R^d\setminus Q_{2\ell'+1}(x')}=A(\X^{x,t})|_{\R^d\setminus Q_{2\ell'+1}(x')}}.
\end{multline*}
\item For all $\sigma(A)$-measurable random variables $Z=Z(A)$ with $0<\sigma^2:=\var{Z}<\infty$, we have
\begin{equation}\label{eq:chat-ar-gen-Kol}
\qquad\dK\big(\sigma^{-1}({Z-\expec{Z}}),\Nc\big)\lesssim\operatorname{RHS}_{\eqref{eq:chat-ar-gen}}(Z)+G_1(Z),
\end{equation}
where $\operatorname{RHS}_{\eqref{eq:chat-ar-gen}}(Z)$ denotes the RHS of~\eqref{eq:chat-ar-gen}, and where we have set
\begin{equation*}
G_1(Z)\,:=\,\sigma^{-3} \inf_{0<\lambda<1}\sum_{x}\sum_t\bigg(\sum_{\ell=1}^\infty \pi(t,\ell)^\lambda\,\expec{\Big(\pardis{\ell,x,t}Z\Big)^{\frac6{1-\lambda}}}^{1-\lambda}\bigg)^\frac12.
\end{equation*}
\end{enumerate}
If in addition for all $x,t$ there exists a $\sigma(\X|_{E_{x,t}},\X'|_{E_{x,t}})$-measurable action radius $\rho_{x,t}$ for $A(\X)$ wrt $\X$ on $E_{x,t}$, then we simply have $\tilde\rho_{x,t}=\rho_{x,t}$ for all $x,t$, the weights $\pi^\frac13$ and $\pi^\frac12$ can both be replaced by $\pi$ in the first two RHS terms of~\eqref{eq:chat-ar-gen} and in the corresponding terms in $\operatorname{RHS}_{\eqref{eq:chat-ar-gen}}(Z)$ in~\eqref{eq:chat-ar-gen-Kol}, and the term $G_1(Z)$ in \eqref{eq:chat-ar-gen-Kol} can be replaced by
\begin{equation*}
G_2(Z)\,:=\,\sigma^{-3}\inf_{0<\lambda<1}\sum_{x}\sum_t \sum_{\ell=1}^\infty \pi(t,\ell)^\lambda\,\expec{\Big(\pardis{\ell,x,t}Z\Big)^{\frac6{1-\lambda}}}^{\frac{1-\lambda}{2}}.
\qedhere
\end{equation*}
\end{theor}
\begin{rem}
The additional term $G_1(Z)$ in~\eqref{eq:chat-ar-gen-Kol} typically dominates the RHS terms of~\eqref{eq:chat-ar-gen}. However they become of the same order if the weight $\pi$ is super-algebraically decaying, or if the improved form of the above result holds (that is, with $G_1(Z)$ replaced by $G_2(Z)$). In each of the examples considered below, we are in one of these two situations, hence the above bounds on the Kolmogorov and Wasserstein distances essentially coincide.
Otherwise, it might be advantageous to rather bound the Kolmogorov distance by the square-root of the Wasserstein distance and use the above estimate for the latter.
\end{rem}

Before we turn to the proof of~Theorem~\ref{th:chat-ra}, we recall some representative examples analyzed in \cite[Section~3]{DG2} to which it applies.
In each case, we briefly discuss the existence and properties of the action radius $\tilde \rho$, which is a slightly stronger notion of action radius than the one $\rho$ given in Definition~\ref{def:ar} and needed for first-order functional inequalities. For technical details we refer the reader to~\cite[Section~3]{DG2}, where the action radii $\rho$ are constructed.
\begin{enumerate}[(A)]
\item \emph{Poisson random inclusion model.} Consider a Poisson point process $\Pc$ of unit intensity on $\R^d$. For each Poisson point $x\in \Pc$ consider a random radius $r(x)$ (independent of the radii of other points and identically distributed according to some given law $\nu$ on~$\R_+$), and define the inclusion $C_x:=B_{r(x)}(x)$.
Consider the inclusion set $\mathcal I:=\cup_{x\in \Pc} C_x$,
let $A_0,A_1\in \R$ be given values, and define a random field $A$ on $\R^d$ by 
\begin{align*}
A(x)\,:=\,A_0\mathds1_{x\notin \mathcal I}+A_1 \mathds1_{x\in \mathcal I},
\end{align*}
that is, $A$ takes value $A_1$ in the inclusions and $A_0$ outside.
As argued in \cite[Section~3.5]{DG2}, $A$ can be reformulated in the form addressed in Theorem~\ref{th:chat-ra} above with $l=1$, and for all $x,t$ there exists a $\sigma(\X|_{E_{x,t}},\X'|_{E_{x,t}})$-measurable action radius $\rho_{x,t}:=t\,\mathds1_{\X\ne\X^{x,t}}$ (cf.~\cite[proof of Proposition~3.6(i)]{DG2}). The improved form of the above result therefore holds with
\[\pi(t,\ell):=\mathds1_{\ell-1\le t<\ell}\,\pr{\X\ne\X^{0,t}}\le 2\,\nu([t-\tfrac12,t+\tfrac12))\,\mathds1_{\ell-1\le t<\ell}.\]
\item \emph{Random parking process.} Consider the random parking point process $\mathcal R$ with unit radius on $\R^d$ (see Section~\ref{sec:RSA} below for a precise construction based on an underlying Poisson point process $\Pc_0$ of unit intensity on $\R^d\times \R_+$).
As above, for all $x\in \mathcal R$ we denote by $C_x:=B(x)$ the unit spherical inclusion centered at $x$ (so that by definition of $\mathcal R$ all the inclusions are disjoint), we consider the inclusion set $\mathcal I:=\cup_{x\in\mathcal R}C_x$, and we define a random field $A$ on $\R^d$ by
$$
A(x)\,:=\,A_0\mathds1_{x\notin\mathcal I}+A_1\mathds1_{x\in\mathcal I}.
$$
In~\cite[proof of Proposition~3.3]{DG2}, for all $x$ we have constructed an action radius $\rho_x$ wrt the underlying Poisson point process $\Pc_0$ on $Q(x)\times\R_+$.
By definition, this action radius satisfies $\rho_x(\Pc_0^B)\le \rho_x(\Pc_0\cup\Pc_0')$ for all $B\subset\Z^d$: indeed, adding points in the Poisson point process $\Pc_0$ adds possible causal chains, hence increases the defined action radius. Therefore, we deduce $\tilde\rho_x\le\rho_x(\Pc_0\cup\Pc_0')$. As $\Pc_0\cup\Pc_0'$ is itself a Poisson point process on $\R^d\times\R_+$ with doubled intensity, we conclude $\pr{\tilde\rho_x\ge \ell}\le Ce^{-\frac1{C}\ell}$ as in \cite[Proposition~3.3]{DG2}, and we may apply Theorem~\ref{th:chat-ra} with $l=0$ and exponential weight $\pi(\ell)= Ce^{-\frac1{C}\ell}$.
\item \emph{Poisson random tessellations.}
Consider a Poisson point process $\Pc$ on $\R^d$, and let $\mathcal V$ denote the associated Voronoi tessellation of $\R^d$, that is, a partition of $\R^d$ into convex polyhedra $V_x\in \mathcal V$ centered at the Poisson points $x\in \Pc$.
For each point $x\in \Pc$ consider a random value $\alpha(x)$ (independent of the values at other points and identically distributed), and we define a random field $A$ on $\R^d$ by
$$
A(x) \,:=\,\sum_{y\in \Pc} \alpha(y) \mathds1_{x\in V_y}.
$$
As argued in~\cite[proof of Proposition~3.2]{DG2}, $A$ can be reformulated in the form addressed in Theorem~\ref{th:chat-ra} above with $l=0$ and with weight
\begin{align}\label{eq:bound-pi-tilderho-voro}
\pi(\ell)\le\pr{\tilde \rho_x \ge \ell-1}\le Ce^{-\frac1C \ell^d}.
\end{align}
(More precisely, we argue as follows: Denote by $C_i:=\{x\in \R^d:x_i\ge \frac56 |x|\}$, $1\le i\le d$, the $d$ cones in the canonical directions $e_i$ of $\R^d$, and consider the $2d$ cones $C_i^\pm:=\pm (2e_i+C_i)$.  For all $x$, let $\rho_x:=\rho_x^0$ denote the action radius for $A$ defined in~\cite[proof of Proposition~3.2]{DG2}, and let $\tilde\rho_x$ be defined as in the statement of Theorem~\ref{th:chat-ra} above.
By construction, the inequality $\tilde \rho_x \le C L$ holds if for each cone $C_i^\pm$ there exists a cube $Q \subset C_i^\pm \cap \{x:|x_i|\le L\}$ such that $ \Pc_0 \cap Q \ne \varnothing \ne \Pc_0'\cap Q$. By independence of $\Pc_0$ and $\Pc_0'$, and by a union bound, the claim~\eqref{eq:bound-pi-tilderho-voro} follows.)
\end{enumerate}
\begin{proof}[Proof of Theorem~\ref{th:chat-ra}]
{\color{black}Given a $\sigma(A)$-measurable random variable $Z=Z(\X)$, we consider for all $n$ its approximation $Z_n$ as the $\sigma(\X|_{\cup_{|(x,t)|\le n}E_{x,t}})$-measurable random variable $Z_n:=\expeCm{Z}{\X|_{\cup_{|(x,t)|\le n}E_{x,t}}}$.
On the one hand, as $Z_n\to Z$ in $\Ld^2(\Omega)$, we find $\liminf_n\dW(\sigma^{-1}({Z_n-\expec{Z_n}}),\Nc)\ge\dW(\sigma^{-1}({Z-\expec{Z}}),\Nc)$. On the other hand, it is easily checked that
\[\expecm{(\pardis{\ell,x,t} Z_n)^p}\le \expecm{(\pardis{\ell,x,t} Z)^p},\qquad\expecm{(\pardis{\ell,x,t}\pardis{\ell',x',t'} Z_n)^p}\le \expecm{(\pardis{\ell,x,t}\pardis{\ell',x',t'} Z)^p}.\]
Therefore, it suffices to establish the result for all approximations $Z_n$'s uniformly in $n$, while the conclusion follows by approximation.
Let an arbitrary finite set $F\subset\Z^d\times\Z^l$ be fixed and consider a $\sigma(\X|_{\cup_{(x,t)\in F}E_{x,t}})$-measurable random variable $Z=Z(\X)$.
We split the proof into two steps.
In view of the stationarity assumption~(b), the law of the action radius $\rho_{x,t}(\X^B)$ can be chosen independent of $x$, hence similarly for $\tilde\rho_{x,t}$.
}

\medskip

\step1 Application of a result by Chatterjee.\\
By~\cite[Theorem~2.2]{Chat08} (together with the standard spectral gap of Proposition~\ref{prop:sgindep}), we have
\begin{multline}\label{e:Chatterjee-refZ}
\dW\big(\sigma^{-1}({Z-\expec{Z}}),\Nc\big)\lesssim\sigma^{-3}\sum_{x,t}\expec{|\Delta_{x,t}Z|^3}\\
+\sigma^{-2}\bigg(\sum_{x,t}\E\bigg[\Big|\sum_{x',t'}(\Delta_{x,t}\Delta_{x',t'}Z)\overline{\Delta_{x',t'}Z}\Big|^2\bigg]\bigg)^\frac12\\
+\sigma^{-2}\bigg(\sum_{x,t}\E\bigg[\Big|\sum_{x',t'}(\Delta_{x,t}\overline{\Delta_{x',t'}Z}){\Delta_{x',t'}Z}\Big|^2\bigg]\bigg)^\frac12,
\end{multline}
where the sums in $(x,t)$ and $(x',t')$ implicitly run over the finite set $F$, and where we have set 
\begin{gather*}
\Delta_{x,t}Z(\X^B):=Z(\X^B)-Z(\X^{B\cup\{(x,t)\}}),\\
\overline{\Delta_{x,t}Z}:=\sum_{B\subset F\atop (x,t)\notin B}K_B{\Delta_{x,t} Z(\X^B)},\qquad K_B:=\frac{|B|!(|F|-|B|-1)!}{|F|!}.
\end{gather*}
Note that by definition $\sum_{B\subset F: (x,t)\notin B}K_B=1$.
By~\cite[Theorem~4.2]{Lachieze-Peccati-16} (together with the standard spectral gap of Proposition~\ref{prop:sgindep}), the following estimate on the Kolmogorov distance also holds
\begin{multline}\label{e:Chatterjee-Kol-refZ}
\dK\bigg(\frac{Z-\expec{Z}}{\sqrt{\var{Z}}},\Nc\bigg)\lesssim \operatorname{RHS}(Z)+\sigma^{-3}\E\bigg[\Big(\sum_{x,t}|\Delta_{x,t}Z|^2\,\overline{\overline{\Delta_{x,t}Z}}\Big)^2\bigg]^\frac12\\
+\sigma^{-2}\bigg(\sum_{x,t}\E\bigg[\Big|\sum_{x',t'}(\Delta_{x,t}\Delta_{x',t'}Z)\overline{\overline{\Delta_{x',t'}Z}}\Big|^2\bigg]\bigg)^\frac12\\
+\sigma^{-2}\bigg(\sum_{x,t}\E\bigg[\Big|\sum_{x',t'}(\Delta_{x,t}\overline{\overline{\Delta_{x',t'}Z}}){\Delta_{x',t'}Z}\Big|^2\bigg]\bigg)^\frac12,
\end{multline}
where $\operatorname{RHS}(Z)$ stands for the RHS of~\eqref{e:Chatterjee-refZ} above, and
\[\overline{\overline{\Delta_{x,t}Z}}:=\sum_{B\subset F\atop (x,t)\notin B}K_B{|\Delta_{x,t} Z(\X^B)|}.\]

\medskip

\step2 Conditioning wrt the action radius.\\
In this step we reformulate the RHSs of \eqref{e:Chatterjee-refZ} and~\eqref{e:Chatterjee-Kol-refZ} by introducing the action radius $\rho_{x,t}$ for $A$ wrt $\X$. We only address the second RHS term in~\eqref{e:Chatterjee-refZ} since 
all the other terms can be treated similarly. To simplify notation, we write $z:=(x,t)$ and $E_z:=E_{x,t}$.
We start by expanding the square and by distinguishing cases when the differences $\Delta_z$ are taken at the same points,
\begin{eqnarray}\label{eq:pre-chat-decZ}
\lefteqn{\sum_{z}\expec{\bigg|\sum_{z'}(\Delta_z\Delta_{z'}Z)\overline{\Delta_{z'}Z}\bigg|^2}}\nonumber\\
&\le&\sum_{z,z',z''}\expec{|\Delta_z\Delta_{z'}Z||\Delta_z\Delta_{z''}Z||\overline{\Delta_{z'}Z}||\overline{\Delta_{z''}Z}|}\nonumber\\
&=&\sum_{z}\expec{|\Delta_zZ|^2|\overline{\Delta_{z}Z}|^2}
+2\sum_{z\ne z'}\expec{|\Delta_z\Delta_{z'}Z||\Delta_zZ||\overline{\Delta_zZ}||\overline{\Delta_{z'}Z}|}\\
&&+\sum_{z\ne z'} \expec{|\Delta_z\Delta_{z'}Z|^2|\overline{\Delta_{z'}Z}|^2}+\sum_{z,z',z''\atop\text{distinct}} \expec{|\Delta_z\Delta_{z'}Z||\Delta_z\Delta_{z''}Z||\overline{\Delta_{z'}Z}||\overline{\Delta_{z''}Z}|},\nonumber
\end{eqnarray}
where we used the fact that $\Delta_z\Delta_z Z=\Delta_zZ$. We then reformulate the four RHS terms by introducing the action radius.
We only treat the last term in detail (the other terms are similar). Since the product $|\Delta_z\Delta_{z'}Z||\Delta_z\Delta_{z''}Z||\overline{\Delta_{z'}Z}||\overline{\Delta_{z''}Z}|$ vanishes whenever $\X|_{E_z}=\X'|_{E_z}$ or $\X|_{E_{z'}}=\X'|_{E_{z'}}$ or $\X|_{E_{z''}}=\X'|_{E_{z''}}$, we obtain after conditioning wrt the values of $\tilde\rho_{z}$, $\tilde\rho_{z'}$ and $\tilde\rho_{z''}$ (that is, the stronger notion of action radii defined in the statement),
\begin{multline*}
\sum_{z,z',z''\atop\text{distinct}}\expec{|\Delta_z\Delta_{z'}Z||\Delta_z\Delta_{z''}Z||\overline{\Delta_{z'}Z}||\overline{\Delta_{z''}Z}|}\\
\le\sum_{\ell,\ell',\ell''=1}^\infty\sum_{z,z',z''\atop\text{distinct}} \E\Big[|\Delta_{z}\Delta_{z'}Z||\Delta_{z}\Delta_{z''}Z||\overline{\Delta_{z'}Z}||\overline{\Delta_{z''}Z}|\\
\times \mathds1_{\ell-1\le\tilde\rho_{z}<\ell}\,\mathds1_{\X|_{E_z}\ne\X'|_{E_z}}\mathds1_{\ell'-1\le\tilde\rho_{z'}<\ell'}\,\mathds1_{\X|_{E_{z'}}\ne\X'|_{E_{z'}}}\mathds1_{\ell''-1\le\tilde\rho_{z''}<\ell''}\,\mathds1_{\X|_{E_{z''}}\ne\X'|_{E_{z''}}}\Big].
\end{multline*}
Note that the event $\tilde\rho_{z}<\ell$ entails by definition $A(\X^B)|_{\R^d\setminus Q_{2\ell+1}(x)}=A(\X^{B\cup\{z\}})|_{\R^d\setminus Q_{2\ell+1}(x)}$ for all $B\subset F$.
By Hölder's inequality and by definition of $\pardis{}$ and $\pardis{}\pardis{}$, we then obtain for all $0<\lambda<1$,
\begin{multline*}
\sum_{z,z',z''\atop\text{distinct}}\expec{|\Delta_{z}\Delta_{z'}Z||\Delta_{z}\Delta_{z''}Z||\overline{\Delta_{z'}Z}||\overline{\Delta_{z''}Z}|}\le\sum_{\ell,\ell',\ell''=1}^\infty\sum_{z,z',z''\atop\text{distinct}}\sum_{B'\subset F\atop z'\notin B'}K_{B'}\sum_{B''\subset F\atop z''\notin B''}K_{B''}\\
\times\expec{\mathds1_{\ell-1\le\tilde\rho_{z}<\ell}\,\mathds1_{\X|_{E_{z}}\ne\X'|_{E_z}}\mathds1_{\ell'-1\le\tilde\rho_{z'}<\ell'}\,\mathds1_{\X|_{E_{z'}}\ne\X'|_{E_{z'}}}\mathds1_{\ell''-1\le\tilde\rho_{z''}<\ell''}\,\mathds1_{\X|_{E_{z''}}\ne\X'|_{E_{z''}}}}^\lambda\\
\times\expec{\Big(\big|\pardis{\ell,z}\pardis{\ell',z'}Z(\X)\big|\big|\pardis{\ell,z}\pardis{\ell'',z''}Z(\X)\big|\big|{\pardis{\ell',z'}Z(\X^{B'})}\big|\big|{\pardis{\ell'',z''}Z(\X^{B''})}\big|\Big)^\frac1{1-\lambda}}^{1-\lambda}.
\end{multline*}
Again applying Hölder's inequality, noting that $\sum_{B\subset F:z\notin B}K_B=1$, and recalling that $\X$ and $\X^B$ have the same law for all $B\subset F$, we conclude
\begin{multline}\label{eq:bound-lastterm-chat}
\sum_{z,z',z''\atop\text{distinct}}\expec{|\Delta_{z}\Delta_{z'}Z||\Delta_{z}\Delta_{z''}Z||\overline{\Delta_{z'}Z}||\overline{\Delta_{z''}Z}|}\\
\le\sum_{\ell,\ell',\ell''=1}^\infty\sum_{z,z',z''}\big(\pi(t,\ell)\pi(t',\ell')\pi(t'',\ell'')\big)^\frac\lambda3\,\expec{\big|\pardis{\ell,z}\pardis{\ell',z'}Z\big|^\frac{4}{1-\lambda}}^\frac{1-\lambda}4\expec{\big|\pardis{\ell,z}\pardis{\ell'',z''}Z\big|^\frac4{1-\lambda}}^\frac{1-\lambda}4\\
\times\expec{\big|{\pardis{\ell',z'}Z}\big|^\frac4{1-\lambda}}^\frac{1-\lambda}4\expec{\big|{\pardis{\ell'',z''}Z}\big|^\frac4{1-\lambda}}^\frac{1-\lambda}4.
\end{multline}
The other terms in~\eqref{eq:pre-chat-decZ} can be treated similarly, and the results~(i)--(ii) follow.
Finally note that if for all $z$ there is an action radius $\rho_{z}$ for $A$ wrt $\X$ on $E_z$ that is $\sigma(\X|_{E_z},\X'|_{E_z})$-measurable, then the complete independence of $\X$ ensures that $\tilde\rho_z$, $\tilde\rho_{z'}$ and $\tilde\rho_{z''}$ are independent for $z,z',z''$ distinct, so that we simply obtain
\begin{multline*}
\expec{\mathds1_{\ell-1\le\tilde\rho_{z}<\ell}\,\mathds1_{\X|_{E_z}\ne\X'|_{E_z}}\mathds1_{\ell'-1\le\tilde\rho_{z'}<\ell'}\,\mathds1_{\X|_{E_{z'}}\ne\X'|_{E_{z'}}}\mathds1_{\ell''-1\le\tilde\rho_{z''}<\ell''}\,\mathds1_{\X|_{E_{z''}}\ne\X'|_{E_{z''}}}}\\
=\pi(t,\ell)\pi(t',\ell')\pi(t'',\ell'').
\end{multline*}
The exponent $\frac13$ can then be removed from the weights in~\eqref{eq:bound-lastterm-chat}, and the corresponding improved result follows.
\end{proof}
\endgroup

\subsection{Deterministic localization}\label{sec:det-loc}
This section is devoted to the example of Gaussian random fields.
{\color{black}The multiscale second-order Poincaré inequality that we establish is a reformulation of the corresponding result by~\cite{NPR-09} in the context of Malliavin calculus, once the connection between functional and Malliavin derivative is established. The proof is postponed to Appendix~\ref{append-Gaus}.
Alternatively, a direct proof could be obtained by deforming the standard second-order Poincar\'e inequality for i.i.d.\@ Gaussian sequences due to Chatterjee~\cite{Chat-09}.
Before stating the result, we give a suitable definition of the iterated functional derivative~$(\parfct{})^2$:
for all bounded Borel subsets $F,F'\subset\R^d$,
\[(\parfct{A})^2_{F\times F'}Z(\Aa)\,:=\,\int_{F\times F'}\Big|\frac{\partial^2\tilde Z(A)}{\partial A^2}(x,y)\Big|\,dxdy,\]
which is alternatively characterized by
\begin{multline*}
(\parfct{A})^2_{F\times F'}Z(\Aa)\\
\,:=\,
\sup\bigg\{\limsup_{t\downarrow0}\frac{\tilde Z(\Aa+t(\bb+\bb'))-\tilde Z(\Aa+t \bb)-\tilde Z(\Aa+t \bb')+\tilde Z(\Aa)}{t^2}\,:\\
 \qquad \,\sup|\bb|,\sup|\bb'|\le1,\,\supp\bb\subset F, \supp \bb' \subset F'\bigg\}.
\end{multline*}}
{\color{black}
\begin{theor}\label{th:fctinequ}
Let $G$ be an $\R^k$-valued stationary centered Gaussian random field on $\R^d$, characterized by its covariance function $\Cc(x):=\cov{G(x)}{G(0)}$. Decompose $\Cc=\Cc_0\ast\Cc_0$ and assume that for some $0<\beta\ne d$,
\begin{align}\label{eq:ass-C0}
|\Cc_0(x)|\lesssim(1+|x|)^{-\frac12(d+\beta)}
\end{align}
(hence $|\Cc(x)|\lesssim(1+|x|)^{-\beta}$).
Let $\Aa(x):=h(G(x))$ for some $h\in C^2(\R^k)$ with $\nabla h,\nabla^2h\in\Ld^\infty(\R^k)$.
Then for all $\sigma(\Aa)$-measurable random variables~$Z=Z(\Aa)\in\Ld^2(\Omega)$ with $0<\sigma^2:=\var{X}<\infty$ we have for all $1\le p\le\frac{d}{(d-\beta)_+}$,
\begin{multline*}
\dW\big(\sigma^{-1}(Z-\expec{Z})\,,\,\Nc\big)+\dTV{ \sigma^{-1}(Z-\expec{Z})\,}{\,\Nc}\\
\,\lesssim\,\sigma^{-2}\,\expec{\3\parfct{A}Z\3_\beta^4}^\frac14\Big(\expec{\|(\parfct{A})^2X\|_{\beta,\beta}^4}^\frac14+\expec{\|\parfct{A}Z\|_p^4}^\frac14\Big),
\end{multline*}
where the multiplicative constant depends only on $\|h\|_{W^{2,\infty}}$ and on $\Cc$, and where the norms are defined by
\begin{eqnarray*}
\|\parfct{A} X\|_p^p&:=&\int_{\R^d}|\parfct{\Aa,B(z)}Z|^p\,dz,\\
\3\parfct{A} X\3_\beta^2&:=&\int_0^\infty \int_{\R^d}|\parfct{\Aa,B_{\ell+1}(z)}Z|^2\,dz\,(\ell+1)^{-d-\beta-1}\,d\ell,\\
\3(\parfct{A})^2X\3_{\beta,\beta}^2&:=&\int_0^\infty\int_0^\infty\iint_{\R^d\times \R^d} |(\parfct{A})^2_{B_{\ell+1}(z)\times B_{\ell'+1}(z')}Z|^2\,dzdz'\\
&&\hspace{4cm}\times(\ell+1)^{-d-\beta-1}\,d\ell\,(\ell'+1)^{-d-\beta-1}\,d\ell'.
\end{eqnarray*}
If the covariance function $\Cc$ is integrable (that is, $\beta>d$), the above is reduced to the following, for all $1\le p\le\infty$,
\begin{multline*}
\dW\big(\sigma^{-1}(X-\expec{X})\,,\,\Nc\big)+\dTV{ \sigma^{-1}(X-\expec{X})\,}{\,\Nc}\\
 \lesssim\, \sigma^{-2}\expec{\|\parfct{\Aa} X\|_2^4}^\frac14\Big(\expec{\|(\parfct{\Aa})^2 X\|_2^4}^\frac14+\expec{\|\parfct{\Aa} X\|_{p}^4}^\frac14\Big),
\end{multline*}
where 
\begin{equation*}
\|(\parfct{A})^2Z\|_2^2~:=~\iint_{\R^d\times \R^d} |(\parfct{A})^2_{B(z)\times B(z')}Z|^2\,dzdz'.\qedhere
\end{equation*}
\end{theor}

\begin{rem}
The decay assumption~\eqref{eq:ass-C0} formulated in terms of $\Cc_0$ can alternatively be replaced by requiring both $|\Cc(x)|\lesssim(1+|x|)^{-\beta}$ and $||\nabla|^\e\Cc(x)|\lesssim(1+|x|)^{-\beta}$ for some $\e>0$.
Indeed, we can then write
\begin{multline*}
\int_{\R^d}\sup_{B(x)}|\Cc_0\ast\xi|^2\,dx\,\lesssim\,\int_{\R^d}\big(|\Cc_0\ast\xi|^2+||\nabla|^\e\Cc_0\ast\xi|^2\big)\\
\,\lesssim\,\iint_{\R^d}\xi(x)\xi(y)\big(\Cc(x-y)+|\nabla|^{2\e}\Cc(x-y)\big)\,dxdy\,\lesssim\,\iint_{\R^d}|\xi(x)||\xi(y)|(1+|x-y|)^{-\beta}\,dxdy.
\end{multline*}
In view of the Hardy-Littlewood-Sobolev inequality, the RHS is bounded by the $\Ld^{2d/(\beta\wedge d)}$ norm of $\xi$, as used in the proof of Theorem~\ref{th:fctinequ}.
\end{rem}}

%%%%%%%%%%%%%%%%%%%%%%%%%%%%%%%%
%%%%%%%%%%%%%%%%%%%%%%%%%%%%%%%%
%%%%%%%%%%%%%%%%%%%%%%%%%%%%%%%%

\section{Application to spatial averages and to RSA models}\label{chap:appl}

{\color{black}\subsection{Spatial averages of the random field}\label{chap:approx-norm}
In order to illustrate how multiscale second-order Poincar\'e inequalities can be used, we investigate the approximate normality of spatial averages $Z_L:=\fint_{Q_L}(A-\expec{A})$ of the random field.}
For simplicity, we focus on two examples: Gaussian processes and Poisson random inclusions with (unbounded) random radii.
{\color{black}Although for Gaussian fields this result is standard and optimal, the result for the Poisson inclusion model improves on the existing literature~\cite{Baddeley-80,Mase-82,Heinrich-05,HLS-16} (cf.~discussion below).}

\begin{prop}\label{prop:as-no}
We consider the two examples separately.
\begin{enumerate}[(i)]
\item {\color{black}Let $G$ be a stationary Gaussian random field on $\R^d$, characterized by its covariance $\Cc(x):=\cov{G(x)}{G(0)}$. Decompose $\Cc=\Cc_0\ast\Cc_0$ and assume that for some $0<\beta\ne d$,
\[|\Cc_0(x)|\lesssim(1+|x|)^{-(d+\beta)/2}\]
(hence $|\Cc(x)|\lesssim(1+|x|)^{-\beta}$).
Let $\Aa(x):=h(G(x))$ for some $h\in C^2(\R^k)$ with $\nabla h,\nabla^2h\in\Ld^\infty(\R^k)$.
Then the results in~\cite[Proposition~1.5]{DG1} ensure that the rescaled random variable
$Z_L:=L^{\frac12(d\wedge\beta)}X_L$
satisfies $\sigma_L^2:=\var{Z_L}\lesssim 1$. Moreover we have for all $L\ge1$,
\begin{equation}\label{e.as-no-TV}
\dW\bigg(\frac{Z_L}{\sigma_L},\Nc\bigg)+\dTV{\frac{Z_L}{\sigma_L}}{\Nc} \, \lesssim \, \sigma_L^{-2}L^{-\frac 12(d\wedge\beta)}.
\end{equation}}
\item 
Let the random field $A$ be given by the Poisson spherical inclusion model with radius law $\nu$ (cf.\@ example (A) in Section~\ref{sec:ran-loc}), and assume that $\nu$ satisfies for some $\beta>0$,
\[\gamma(\ell):=\nu([\ell,\ell+1))\lesssim \ell^{-3d-\beta-1}.\]
Then the results in~\cite[Proposition~1.5]{DG1} hold with weight $\pi(\ell)=(\ell+1)^{-2d-\beta-1}$ and the rescaled random variable $Z_L:=L^{d/2}X_L$ satisfies $\sigma_L^2:=\var{Z_L}\lesssim 1$. Moreover we have for all $L\ge 1$,
\begin{equation*}
\dW\bigg(\frac{Z_L}{\sigma_L},{\Nc}\bigg)+\dK\bigg(\frac{Z_L}{\sigma_L},{\Nc}\bigg) \, \lesssim \,\sigma_L^{-3}L^{-\frac12(d\wedge\beta)}.
\qedhere
\end{equation*}
\end{enumerate}
\end{prop}
A similar result as above holds in stochastic homogenization, where $Z_L$ is replaced by the spatial average of the homogenization commutator~\cite{DGO1,DO1}.
Estimating $\sigma_L\lesssim 1$ from below is viewed as a separate issue.
In the context of stochastic homogenization with Gaussian coefficients, possible degeneracy (that is, $\sigma_L\ll1$) was first observed in the 1D setting~\cite{Taqqu,Gu-Bal-12,LNZH-17}, and we refer to~\cite{DFG} for a general study of degeneracy issues in the multidimensional setting.

{\color{black}
We briefly compare item~(ii) above with the existing literature for the Poisson inclusion model.
Let the radius law $\nu$ satisfy $\nu([\ell,\ell+1))\lesssim\ell^{-d-\kappa-1}$ with $\kappa>0$.
Previous works on general Boolean models~\cite{Baddeley-80,Mase-82} show that a qualitative CLT holds whenever $\kappa>0$.
Quantitative results are much more recent: a quantitative CLT with optimal rate~$L^{-\frac d2}$ in Kolmogorov distance was first proved in~\cite{Heinrich-05} when the radius law $\nu$ is exponentially decaying, while in~\cite{HLS-16} it is shown that a quantitative CLT (in a weaker metric) holds with rate $L^{-(\frac\kappa2-d)\wedge\frac d2\wedge1}$ whenever $\kappa>2d$. The above improves on these previous results: indeed, we establish a quantitative CLT in Wasserstein and Kolmogorov distances with rate~$L^{-(\frac\kappa2-d)\wedge\frac d2}$ whenever $\kappa>2d$, which we believe to be optimal.}

\begin{proof}[Proof of Proposition~\ref{prop:as-no}]
We split the proof into two steps.

\medskip

\step1 Proof of item~(i).\\
{\color{black}By~\cite[Theorem~3.1]{DG2}, we may apply \cite[Proposition~1.5]{DG1} with weight $\pi(\ell):=(\ell+1)^{-\beta-1}$, which yields $\sigma_L\lesssim 1$.
Next, we apply Theorem~\ref{th:fctinequ} to $Z_L$, which greatly simplifies in this precise linear situation since second derivatives of $Z_L$ wrt $A$ vanish identically.
More precisely, for all $L\ge1$, it leads to the following, for all $1\le p\le\frac{d}{(d-\beta)_+}$,
\begin{multline}\label{eq:bound-appl-Gauss}
\dW\bigg(\frac{Z_L}{\sigma_L},\Nc\bigg)+\dTV{\frac{Z_L}{\sigma_L}}{\Nc}
\,\lesssim\,\sigma_L^{-2}\,\expec{\Big(\int_{\R^d}|\parfct{\Aa,B(z)}Z_L|^p\,dz\Big)^\frac4p}^\frac14\\
\times\expec{\Big(\int_0^\infty \int_{\R^d}|\parfct{\Aa,B_{\ell+1}(z)}Z_L|^2\,dz\,(\ell+1)^{-d-\beta-1}\,d\ell\Big)^2}^\frac14.
\end{multline}
We compute for all $\ell\ge 0$,
\[\parfct{A,B_{\ell+1}(z)}Z_L(A)\, =\,L^{\frac12(d\wedge\beta)-d} |B_{\ell+1}(z) \cap Q_L|.\]
Choosing $p=\frac{d}{(d-\beta)_+}$, the first RHS factor in~\eqref{eq:bound-appl-Gauss} is then estimated by
\[\Big(\int_{\R^d}|\parfct{\Aa,B(z)}Z_L|^p\,dz\Big)^\frac1p\,\lesssim\,L^{\frac12(d\wedge\beta)-d+\frac dp}\,=\,L^{-\frac12(d\wedge\beta)}.\]
It remains to treat the second RHS factor in~\eqref{eq:bound-appl-Gauss}.
Note that it already appears in the proof that $\sigma_L\lesssim 1$ (cf.~\cite{DG1}); we reproduce the short argument for completeness.
We split the integral over $\ell$ into the contribution with $\ell \le L$ and $\ell >L$,
and we note that the integral over $x$ reduces to $B_{L+\ell}$,
to the effect of
\begin{multline*}
\int_0^\infty \int_{\R^d}|\parfct{\Aa,B_{\ell+1}(z)}Z_L|^2\,dz\,(\ell+1)^{-d-\beta-1}\,d\ell\\
\lesssim~L^{d\wedge\beta-2d}\int_0^L \ell^{2d}L^d(\ell+1)^{-d-\beta-1}\,d\ell+L^{d\wedge\beta-2d}\int_L^\infty L^{2d}\ell^d(\ell+1)^{-d-\beta-1}\,d\ell
~\lesssim~1.
\end{multline*}
The conclusion~\eqref{e.as-no-TV} follows.}

\medskip

\step2 Proof of item~(ii).\\
By \cite[Proposition~3.6]{DG2}, we may apply \cite[Proposition~1.5]{DG1} with the weight
\[\pi(\ell)\,\simeq\, (\ell+1)^d\sup_{u\ge \frac1{\sqrt d}\ell-4}\gamma(u)\,\lesssim\,\ell^{-2d-\beta-1},\]
which implies $\pi_*(L)\simeq L^{d}$, hence $\sigma_L\lesssim1$.
We then apply Theorem~\ref{th:chat-ra} to $Z_L$.
For all $x,x'\in \Z^d$ and $\ell,\ell'\in\N$, we have
\[|\pardis{\ell,x}Z_L|\,\lesssim \, L^{-\frac d2} |B_{\ell+1}(x)\cap Q_L|  \,\lesssim\, L^{-\frac d2} \big(L\wedge (\ell+1)\big)^d \mathds1_{|x|\lesssim L+\ell},\]
and also
\begin{align*}
|\pardis{\ell,x}\pardis{\ell',x'}Z_L|\,&\lesssim \, L^{-\frac d2} |B_{\ell'+1}(x')\cap B_{\ell+1}(x)\cap Q_L| \\
&\lesssim \,L^{-\frac d2}  \big(L\wedge (\ell+1)\wedge (\ell'+1)\big)^d \mathds1_{|x'|\lesssim L+\ell'} \mathds1_{|x|\lesssim L+\ell} \mathds1_{|x-x'|\lesssim \ell+\ell'}.
\end{align*}
As these RHS are deterministic, we may actually apply Theorem~\ref{th:chat-ra} with the borderline exponent $\lambda=1$, which yields
\begin{eqnarray*}
\lefteqn{\dW\bigg(\frac{Z_L}{{\sigma_L}},\Nc\bigg)+\dK\bigg(\frac{Z_L}{{\sigma_L}},\Nc\bigg)}\\
&\lesssim_\mu& \frac1{\sigma_L^2} \Bigg(\sum_{x,x',x''}\sum_{\ell,\ell',\ell''=0}^\infty \gamma(\ell)\gamma(\ell')\gamma(\ell'')\supessd{A} |\pardis{\ell,x}\pardis{\ell',x'}Z_L|\\
&&\hspace{2cm}\times \supessd{A} |\pardis{\ell,x}\pardis{\ell'',x''}Z_L|\supessd{A} |\pardis{\ell',x'} Z_L|\supessd{A} |\pardis{\ell'',x''} Z_L|\Bigg)^\frac12\\
&&+\frac1{\sigma_L^2} \Bigg(\sum_{x,x'}\sum_{\ell,\ell'=0}^\infty \gamma(\ell)\gamma(\ell')\supessd{A} |\pardis{\ell,x}\pardis{\ell',x'}Z_L|^2\supessd{A} |\pardis{\ell',x'} Z_L|^2\Bigg)^\frac12\\
&&+\frac1{\sigma_L^2} \Bigg(\sum_{x}\sum_{\ell=0}^\infty \gamma(\ell)\supessd{A} |\pardis{\ell,x}Z_L|^4\Bigg)^\frac12
+\frac1{\sigma_L^{3}}  \sum_{x}\sum_{\ell=0}^\infty \gamma(\ell)\supessd{A} |\pardis{\ell,x}Z_L|^3 .
\end{eqnarray*}
We denote by $I_1,I_2,I_3,I_4$ the four RHS terms. 
Given the bound $\gamma(\ell)\lesssim \ell^{-\beta'-1}$ for some $\beta'>0$, straightforward calculations left to the reader yield for all $L\ge1$,
\begin{align*}
I_1\,&\lesssim \,\sigma_L^{-2}L^{-\frac d2} (1\vee L^{2d-\beta'})^\frac32,
\\
I_2\,&\lesssim \,\sigma_L^{-2}L^{-\frac d2} (1\vee L^{3d-\beta'})^\frac12 (1\vee L^{2d-\beta'})^\frac12,
\\
I_3\,&\lesssim \,\sigma_L^{-2} L^{-\frac d2} (1\vee L^{4d-\beta'})^\frac12,
\\
I_4\,&\lesssim \,\sigma_L^{-3} L^{-\frac{d}2} (1\vee L^{3d-\beta'}).
\end{align*}
The dominating term wrt scaling in $L$ is the third one $I_3$, and the conclusion then follows by taking $\beta':=3d+\beta$ for $\beta>0$.
\end{proof}

\subsection{Random sequential adsorption and the jamming limit}\label{sec:RSA}

We consider the problem of sequential packing at saturation, following the presentation in \cite{Schreiber-Penrose-Yukich-07}.
Let $R>0$, and let $(U_{i,R})_{i\ge1}$ be a sequence of i.i.d.\@ random points uniformly distributed on the cube~$Q_R$. Let $\calS$ be a fixed bounded closed
convex set in $\R^d$ with non-empty interior and centered at the origin $0$ of
$\R^d$ (that is, a reference ``solid''), and for $i \ge1$ let $\calS_{i,R}$ be the translate of $\calS$ with center
at $U_{i,R}$. Then $\calS_R := (\calS_{i,R})_{i\ge 1}$ is an infinite sequence of solids centered at uniform random positions in $Q_R$ (the centers lie in $Q_R$ but the solids themselves need not lie wholly
inside $Q_R$).
Let the first solid $S_{1,R}$ be packed, and recursively for $i \ge2$ let the $i$-th solid
$\calS_{i,R}$ be packed if it does not overlap any solid in $\{\calS_{1,R},\dots, \calS_{i-1,R}\}$
which has already been packed. If not packed, the $i$-th solid is discarded.
This process, known as random sequential adsorption (RSA) with
infinite input on the domain $Q_R$, is irreversible and terminates when it is not possible to accept additional
solids. 
The jamming number $\Nc_R := \Nc_R(\calS_R)$ denotes the number of solids packed in $Q_R$ at termination. We are then interested in the asymptotic 
behavior of $R^{-d}\Nc_R$ in the infinite volume regime $R\uparrow\infty$, the limit of which (if it exists) is called the \emph{jamming limit}. 

\medskip

In any dimension $d \ge1$ and for any choice of the reference solid $\calS$, Penrose \cite{Penrose-01} established the existence of the jamming limit, as well as the existence of the infinite volume limit for the distribution of the centers of packed solids, which defines a point process $\xi$ on the whole of $\R^d$.
(In the model case $\calS:=B_1$, this process $\xi$ is referred to as the random parking point process with unit radius.)
As we now quickly recall, the key argument in~\cite{Penrose-01} relies on a graphical construction for $\xi$ as a transformation $\xi=\Phi(\Pc_0)$ of a unit intensity Poisson point process $\Pc_0$ on the extended space $\R^d\times\R_+$.
We first construct an oriented graph on the points of $\Pc_0$ in $\R^d\times\R_+$, by putting an oriented edge from $(x,t)$ to $(x',t')$ whenever $(x+\calS)\cap (x'+\calS)\ne\varnothing$ and $t<t'$ (or $t=t'$ and $x$ precedes $x'$ in the lexicographic order, say). We say that $(x',t')$ is an offspring (resp.\@ a descendant) of $(x,t)$, if $(x,t)$ is a direct ancestor (resp.\@ an ancestor) of $(x',t')$, that is, if there is an edge (resp.\@ a directed path) from $(x,t)$ to $(x',t')$. The set $\xi:=\Phi(\Pc_0)$ is then constructed as follows. Let $F_1$ be the set of all roots in the oriented graph (that is, the points of $\Pc_0$ without ancestor), let $G_1$ be the set of points of $\Pc_0$ that are offsprings of points of $F_1$, and let $H_1:=F_1\cup G_1$. Now consider the oriented graph induced on $\Pc_0\setminus H_1$, and define $F_2,G_2,H_2$ in the same way, and so on. By construction, the sets $(F_j)_j$ and $(G_j)_j$ are all disjoint and constitute a partition of $\Pc_0$. We finally define $\xi:=\Phi(\Pc_0):=\bigcup_{j=1}^\infty F_j$.

\medskip

In \cite{Schreiber-Penrose-Yukich-07}, Schreiber, Penrose, and Yukich further showed in any dimension $d\ge1$ that the rescaled variance $R^{-d}\var{\Nc_R}$ converges to
a positive limit (without rate) and that $\Nc_R$ satisfies a CLT, that is, the fluctuations of
the random variable $\Nc_R$ are asymptotically Gaussian. They also quantified the rate of
convergence to the Gaussian, as well as the rate of convergence of $R^{-d}\expecm{\Nc_R}$ to the jamming
limit. The numerical approximation of the value of the jamming limit has been the object of
several works, including \cite[Chapter~11.4]{Torquato-02} and \cite{TUS-06}. As is clear from the analysis, the speed of convergence of $R^{-d}\expecm{\Nc_R}$ towards its limit is dominated by a boundary effect (the error scales like $R^{-1}$).

\medskip
 
In order to avoid this boundary effect and to obtain better rates of convergence, we may replace $\Nc_R$ by the number $\tilde\Nc_R$ of packed solids with periodic boundary conditions on $Q_R$: we say that the $i$-th solid $\calS_{i,R}$ is packed with periodic boundary conditions if its {\it periodic extension} $\calS_{i,R}+R\Z^d$ does not overlap with any solid in $\{\calS_{1,R},\ldots,\calS_{i-1,R}\}$ which has already been packed. The following shows that this procedure allows to get rid of the boundary effect, yields optimal estimates, and therefore suggests a more efficient way to approximate the jamming limit numerically.
\begin{theor}\label{t.RPM}
For all $R\ge 0$, let $\tilde\Nc_R :=\tilde\Nc_R(\calS_{R})$ be the number of packed solids of $\calS_{R}$ with periodic boundary conditions as defined above.
There are constants $\mu :=\mu(\calS,d)\in (0,\infty)$ (the jamming limit) and $\sigma^2:=\sigma^2(\calS,d)\in (0,\infty)$ such that as $R\uparrow \infty$ we have
\begin{gather}\label{e.estim-RPM-1}
|R^{-d}\expecm{\tilde\Nc_R}-\mu|~\lesssim~ e^{-\frac1CR}, \\
|R^{-d}\varm{\tilde\Nc_R}-\sigma^2|~\lesssim~ e^{-\frac1CR},\label{e.estim-RPM-2}
\end{gather}
and
\begin{equation}\label{e.CLT-RPM}
\dW\Big(R^\frac d2 (R^{-d} \tilde\Nc_R-\mu)\,,\,\Nc(\sigma^2)\Big)+\dK\Big(R^\frac d2 (R^{-d} \tilde\Nc_R-\mu)\,,\,\Nc(\sigma^2)\Big)\,\lesssim\, R^{-\frac d2},
\end{equation}
where $\Nc(\sigma^2)$ denotes a centered normal random variable with variance $\sigma^2$.
\end{theor}
Estimates~\eqref{e.estim-RPM-1} and~\eqref{e.estim-RPM-2} are obtained as consequences of the stabilization properties established in~\cite{Schreiber-Penrose-Yukich-07}.
Note that \eqref{e.CLT-RPM} is the best one can hope for: If we considered a Poisson point process instead of the random parking process, then $\tilde\Nc_R$ would be
the number of Poisson points in $Q_R$, $\mu$ would be the intensity of the process, we would have $\sigma^2=\mu$, and~\eqref{e.CLT-RPM} would be sharp.
The proof of  \eqref{e.CLT-RPM} combines~\eqref{e.estim-RPM-1} and~\eqref{e.estim-RPM-2} to a normal approximation result, which is itself a slight improvement 
of \cite[Theorem~1.1]{Schreiber-Penrose-Yukich-07} in the sense that it avoids the spurious logarithmic correction $\log^{3d}(R)$.
This improvement is a direct consequence of Theorem~\ref{th:chat-ra}. It also follows from~\cite[Theorem~6.1]{LPS-16}, but the proof we display here is more direct.

\begin{proof}[Proof of Theorem~\ref{t.RPM}]
Denote by $\xi_R$ the ($R$-periodic extension of the) random parking measure on $Q_R$ with periodic boundary conditions (that is, the measure obtained as the sum of Dirac masses at the centers of the periodically packed solids in $Q_R$). Also denote by $\xi=\xi_\infty$ the corresponding random parking measure on the whole space $\R^d$. Note that by definition both measures $\xi_R$ and $\xi$ are stationary, and we have $\xi_R(Q_R)=\tilde\Nc_R$.

\smallskip\noindent
Let us first introduce a natural pairing between $\xi_R$ and $\xi$ based on the graphical construction recalled above. 
Replacing the original Poisson point process $\Pc_0$ by $\Pc_0 \cap (Q_R\times \R_+)+R\Z^d$ (that is, the $R$-periodization of the restriction of $\Pc_0$ to $Q_R\times\R_+$), and then running the same graphical 
construction as above, we obtain a version of the $R$-periodic random parking measure $\xi_R$. Using this version, we view both $\xi_R$ and $\xi$ as $\sigma(\Pc_0)$-measurable random measures for the same underlying Poisson point process $\Pc_0$. Note however that with this coupling the pair $(\xi_R,\xi)$ is no longer stationary.

\smallskip\noindent
We split the proof into three steps.
In the first step we recall the construction of action radii for $\xi_R$ and $\xi$. We then prove \eqref{e.estim-RPM-1} and~\eqref{e.estim-RPM-2} using the exponentially decaying tail  of the constructed action radii (or alternatively, the multiscale covariance inequality of \cite[Proposition~3.3]{DG2}), and finally we prove~\eqref{e.CLT-RPM} by appealing to Theorem~\ref{th:chat-ra}.

\medskip

\step1 Construction and properties of action radii.\\
In this step we claim for all $y$ that $\xi$ admits an action radius $\rho_y$ wrt $\Pc_0$ on $Q(y)\times\R_+$, that the restriction $\xi_R|_{Q_R}$ admits an action radius $\rho_{R,y}$ wrt $\Pc_0$ on $Q(y)\times\R_+$, and that we have
\[\pr{\rho_y>\ell}+\pr{\rho_{R,y}>\ell}\lesssim e^{-\frac1C\ell}.\]
In particular, we show that this implies
\begin{align}\label{e.stab-estim}
\sup_{y\in Q_{R/2}}\pr{\xi(Q(y))\ne\xi_{R}(Q(y))}\,\lesssim \,e^{-\frac1CR}.
\end{align}
The construction and tail behavior of the action radius $\rho_y$ follows from \cite[Proposition~3.3]{DG2} (with $\ell=0$). Let the action radius $\rho_{R,y}$ be constructed similarly (simply replacing $\Pc_0$ by the point set $\Pc_0\cap(Q_R\times\R_+)+R\Z^d$). A careful inspection of the proof of~\cite[Lemma~3.5]{Schreiber-Penrose-Yukich-07} reveals that the same exponential tail behavior holds for $\rho_{R,y}$ uniformly in $R>0$.
It remains to argue in favor of \eqref{e.stab-estim}, which simply follows from the exponential tail behavior of the action radii in the form
\begin{align*}
\sup_{y\in Q_{R/2}}\pr{\xi(Q(y))\ne\xi_{R}(Q(y))}\,\le\,\sup_{y\in Q_{R/2}}\pr{Q(y)+B_{\rho_y}\not\subset Q_R}\,\lesssim \,e^{-\frac1CR}.
\end{align*}
Note that~\eqref{e.stab-estim} also holds in the following form: for all $0<r<1$, there exists $C_r>0$ such that for all $R\ge 1$,
\[\sup_{y\in Q_{rR}}\pr{\xi(Q(y))\ne\xi_{R}(Q(y))}\,\lesssim \,e^{-\frac1{C_r}R}.\]
Using the uniform bound $\xi_R(Q_R \setminus Q_{rR})\lesssim |Q_R \setminus Q_{rR}|$, and letting $R\uparrow \infty$ and then $r\uparrow 1$, we deduce that
$\lim_{R\uparrow \infty} R^{-d}\expecm{\tilde\Nc_R} =\lim_{R\uparrow \infty} R^{-d}\expecm{\Nc_R}$ coincides with the jamming limit.

\medskip

\step2 Proof of \eqref{e.estim-RPM-1} and~\eqref{e.estim-RPM-2}.\\
By stationarity of $\xi_R$ and $\xi$ we find $\expec{\xi_R(Q_R)}=R^d\expec{\xi_R(Q)}$ and $\expec{\xi(Q_R)}=\mu R^d$ with $\mu:=\expec{\xi(Q)}$. We define 
\begin{align}\label{eq:int-cov-xi}
\sigma^2\,:=\,\int_{\R^d} \cov{\xi(Q(x))}{\xi(Q)}dx
\end{align}
and shall prove~\eqref{e.estim-RPM-1} and~\eqref{e.estim-RPM-2} in the form
\begin{equation}\label{e.conv-varianceRPM}
|R^{-d}\expecm{\tilde\Nc_R}-\mu|\lesssim e^{-\frac1CR}\qquad\text{and}\qquad |R^{-d}\varm{\tilde\Nc_R}-\sigma^2| \lesssim e^{-\frac1CR}.
\end{equation}
The estimate for the convergence of the mean follows from~\eqref{e.stab-estim} in the form
\begin{multline*}
|R^{-d}\expecm{\tilde\Nc_R}-\mu|=|\expecm{\xi_R(Q)-\xi(Q)}|\\
\le \supess\big(\xi_R(Q)+\xi(Q)\big)\, \pr{\xi_R(Q)\ne\xi(Q)}~\lesssim\,e^{-\frac1CR}.
\end{multline*}
We now appeal to the covariance inequality of \cite[Proposition~3.3]{DG2} to prove both the existence of $\sigma^2$ (by showing that the integral~\eqref{eq:int-cov-xi} is absolutely convergent) 
and the estimate for the convergence of the variance in~\eqref{e.conv-varianceRPM}.
Rather than using the complete covariance inequality, it is actually sufficient here to make direct use of the constructed action radii $\rho_0$ and $\rho_{R,0}$ of Step~1.
For $|y|\ge \sqrt d+1$, noting that given $\rho_0\vee\rho_y\le\frac12(|y|-\sqrt d)$ the random variables $\xi(Q(y))$ and $\xi(Q)$ are by definition independent, we obtain
\begin{eqnarray*}
\lefteqn{\cov{\xi(Q(y))}{\xi(Q)}}\\
&=&\expecm{(\xi(Q(y))-\mu)(\xi(Q)-\mu)\mathds1_{\rho_0\vee\rho_y> \frac12(|y|-\sqrt d)}}\\
&&\quad+\expeCm{(\xi(Q(y))-\mu)(\xi(Q)-\mu)}{\rho_0\vee\rho_y\le\tfrac12(|y|-\sqrt d)}\,\prm{\rho_0\vee\rho_y\le\tfrac12(|y|-\sqrt d)}\\
&=&\expecm{(\xi(Q(y))-\mu)(\xi(Q)-\mu)\mathds1_{\rho_0\vee\rho_y> \frac12(|y|-\sqrt d)}}\\
&&\quad+\prm{\rho_0\vee\rho_y\le \tfrac12(|y|-\sqrt d)}^{-1}\\
&&\hspace{2.5cm}\times\expecm{(\xi(Q(y))-\mu)\mathds1_{\rho_0\vee\rho_y\le\frac12(|y|-\sqrt d)}}\,\expecm{(\xi(Q)-\mu)\mathds1_{\rho_0\vee\rho_y\le\frac12(|y|-\sqrt d)}}\\
&=&\expecm{(\xi(Q(y))-\mu)(\xi(Q)-\mu)\mathds1_{\rho_0\vee\rho_y> \frac12(|y|-\sqrt d)}}\\
&&\quad+\big(1-\prm{\rho_0\vee\rho_y> \tfrac12(|y|-\sqrt d)}\big)^{-1}\\
&&\hspace{2.5cm}\times\expecm{(\xi(Q(y))-\mu)\mathds1_{\rho_0\vee\rho_y>\frac12(|y|-\sqrt d)}}\,\expecm{(\xi(Q)-\mu)\mathds1_{\rho_0\vee\rho_y>\frac12(|y|-\sqrt d)}},
\end{eqnarray*}
and hence, for all $|y|\ge C$ with $C\simeq1$ large enough such that
\[\prm{\rho_0\vee\rho_y>\tfrac12(|y|-\sqrt d)}\le2\,\prm{\rho_0>\tfrac12(|y|-\sqrt d)}\le \tfrac12,\]
we conclude
\begin{align}\label{eq:est-cov-xi}
|\cov{\xi(Q(y))}{\xi(Q)}|&\lesssim \supess\big(\xi(Q)^2\big)\,\prm{\rho_0> \tfrac12(|y|-\sqrt d)}\lesssim e^{-\frac1C|y|}.
\end{align}
Arguing similarly for $\xi_R$ with $\rho_0$ replaced by $\rho_{R,0}$, we deduce for all $y\in Q_R$,
\begin{align}\label{eq:est-cov-xiR}
|\cov{\xi_R(Q(y))}{\xi_R(Q)}|&\lesssim e^{-\frac1C|y|}.
\end{align}
The estimate~\eqref{eq:est-cov-xi} implies in particular that the integral for $\sigma^2$ in~\eqref{eq:int-cov-xi} is well-defined.
It remains to prove the estimate for the convergence of the variance in~\eqref{e.conv-varianceRPM}.
By $R$-periodicity and stationarity of $\xi_R$, we find
\begin{align*}
R^{-d}\varm{\tilde\Nc_R} &= R^{-d}\var{\int_{Q_R}\xi_R(Q(y))dy}\\
&= \fint_{Q_R} \int_{Q_R} \cov{\xi_R(Q(x-y))}{\xi_R(Q)}dxdy\\
&= \int_{Q_R} \cov{\xi_R(Q(y))}{\xi_R(Q)}dy,
\end{align*}
so that we may decompose
\begin{multline}\label{eq:decomp-diff-var}
\sigma^2-R^{-d}\varm{\tilde\Nc_R}\,=\,\int_{\R^d\setminus Q_{R/2}}\cov{\xi(Q(y))}{\xi(Q)}dy-\int_{Q_R\setminus Q_{R/2}}\cov{\xi_R(Q(y))}{\xi_R(Q)}dy\\
+\int_{Q_{R/2}}\big(\cov{\xi(Q(y))}{\xi(Q)}-\cov{\xi_R(Q(y))}{\xi_R(Q)}\big)dy.
\end{multline}
We estimate each of the three RHS terms separately. On the one hand, the estimates~\eqref{eq:est-cov-xi} and~\eqref{eq:est-cov-xiR} yield
\begin{align*}
\Big|\int_{\R^d\setminus Q_{R/2}}\cov{\xi(Q(y))}{\xi(Q)}dy\Big|\lesssim\int_{\R^d\setminus Q_{R/2}}e^{-\frac1C|y|}dy\lesssim e^{-\frac1CR}.
\end{align*}
and
\begin{align*}
\Big|\int_{Q_R\setminus Q_{R/2}}\cov{\xi_R(Q(y))}{\xi_R(Q)}dy\Big|\lesssim \int_{Q_R\setminus Q_{R/2}}e^{-\frac1C|y|}dy\lesssim e^{-\frac1CR}.
\end{align*}
On the other hand, using~\eqref{e.stab-estim}, we obtain
\begin{eqnarray*}
\lefteqn{\Big|\int_{Q_{R/2}}\big(\cov{\xi(Q(y))}{\xi(Q)}-\cov{\xi_R(Q(y))}{\xi_R(Q)}\big)dy\Big|}\\
&\le&\int_{Q_{R/2}}\expec{\,\big|\xi(Q)-\expec{\xi(Q)}\big|\,\big|\xi(Q(y))-\xi_R(Q(y))\big|\,}\\
&&\hspace{3cm}+\int_{Q_{R/2}}\expec{\,\big|\xi_R(Q(y))-\expec{\xi_R(Q(y))}\big|\,\big|\xi(Q)-\xi_R(Q)\big|\,}dy\\
&\lesssim&R^d\supess\big(\xi(Q)+\xi_R(Q)\big)\,\sup_{y\in Q_{R/2}}\pr{\xi(Q(y))\ne\xi_R(Q(y))}~\lesssim\, e^{-\frac1CR}.
\end{eqnarray*}
Injecting these estimates into~\eqref{eq:decomp-diff-var}, the conclusion~\eqref{e.conv-varianceRPM} for the convergence of the variance follows.

\medskip

\step3 Proof of \eqref{e.CLT-RPM}.\\
We claim that it is enough to prove the normal approximation estimate
\begin{equation}\label{e.CLT-RPMb}
\dW\bigg(\frac{\Nc_R-\expec{\Nc_R}}{\sqrt{\var{\Nc_R}}}\,,\,\Nc\bigg)+\dK\bigg(\frac{\Nc_R-\expec{\Nc_R}}{\sqrt{\var{\Nc_R}}}\,,\,\Nc\bigg)\,\lesssim\, R^{-\frac d2}.
\end{equation}
Indeed, the result~\eqref{e.CLT-RPM} then follows from \eqref{e.CLT-RPMb}, \eqref{e.estim-RPM-1}, and \eqref{e.estim-RPM-2} by the triangle inequality.
We omit the proof  of \eqref{e.CLT-RPMb}, which is identical to the proof of Proposition~\ref{prop:as-no}(ii) (the correction $L^{d-\beta}$ disappears here since the weight is exponential).
\end{proof}

%
%%%%%%%%%%%%%

\appendix
\addtocontents{toc}{\protect\setcounter{tocdepth}{0}}
\section{Proof for Gaussian fields} \label{append-Gaus}

{\color{black}
In this appendix, we consider the Gaussian case. Let $G$ be an $\R^k$-valued stationary centered Gaussian random field on $\R^d$, characterized by its covariance function $\Cc(x):=\cov{G(x)}{G(0)}$, and let $\Aa(x):=h(G(x))$ for some $h\in C^2(\R^k)$ with $\nabla h,\nabla^2h\in\Ld^\infty(\R^k)$.
Rather than giving a direct proof of Theorem~\ref{th:fctinequ} based on deforming the standard second-order Poincaré inequality satisfied by i.i.d.\@ Gaussian sequences (as we did for first-order inequalities in~\cite{DG2}), we make direct use of the known results in the framework of Malliavin calculus.
We first briefly recall basic notions of Malliavin calculus wrt the Gaussian field $G$, then we explain how the functional derivative~$\parfct{}$ relates to the Malliavin derivative, and we conclude how to derive multiscale functional inequalities from their known Malliavin counterparts.

\subsection{Reminder on Malliavin calculus}

%Since the covariance function $\Cc$ is bounded, the random field $G$ can be seen as a random Schwartz distribution, that is, as a random element in $\Sc'(\R^d)^k$:
{\color{black}For $\zeta,\zeta'\in C^\infty_c(\R^d)^k$, the random variables $\int_{\R^d} G\zeta$ and $\int_{\R^d} G\zeta'$
are by definition centered Gaussians with covariance
\[{\textstyle\cov{\int_{\R^d} G\zeta}{\int_{\R^d} G\zeta'}}:=\int_{\R^d}\int_{\R^d} \Cc(x-y):\zeta(x)\otimes\zeta'(y)\,dxdy.\]
We define  $\Hf$ as the closure of $C^\infty_c(\R^d)^k$ for the Hilbert norm
\begin{equation}\label{e.malliavin-SP}
\|\zeta\|_{\Hf}^2:=\langle\zeta,\zeta\rangle_\Hf,\qquad\langle\zeta,\zeta'\rangle_\Hf:=\int_{\R^d}\int_{\R^d} \Cc(x-y):\zeta(x)\otimes\zeta'(y)\,dxdy.
\end{equation}}
This space $\Hf$ (up to taking the quotient wrt the kernel of the norm) is a separable Hilbert space.
%Due to the defining relation $\cov{G(\zeta)}{G(\zeta')}=\langle\zeta,\zeta'\rangle_\Hf$, the random field $G$ is viewed as an isonormal Gaussian process over $\Hf$.
We recall some elements of the Malliavin calculus wrt the Gaussian field $G$ (see e.g.~\cite{Malliavin-97,Nourdin-Peccati-12} for details).
Without loss of generality, we may assume that the probability space $\Omega$ is endowed with the minimal $\sigma$-algebra $\F=\sigma(G)$. This implies that the linear subspace
\begin{align}\label{eq:def-Rc}
\Rc:=\Big\{g\big({\textstyle\int_{\R^d} G\zeta_1,\ldots,\int_{\R^d} G\zeta_n}\big)\,:\,n\in\N,\,g\in C_c^\infty((\R^d)^n),\,\zeta_1,\ldots,\zeta_n\in C_c^\infty(\R^d)^k\Big\}
\end{align}
is dense in $\Ld^2(\Omega)$, which allows to define operators and prove properties on the simpler subspace $\Rc$ before extending them to $\Ld^2(\Omega)$ by density. For $r\ge1$, we similarly define
\[\Rc(\Hf^{\otimes r}):=\Big\{\sum_{i=1}^n\psi_iZ_i\,:\,n\in\N,\,Z_1,\ldots, Z_n\in\Rc,\,\psi_1,\ldots,\psi_n\in \Hf^{\otimes r}\Big\}\,\subset\,\Ld^2(\Omega;\Hf^{\otimes r}),\]
which is dense in $\Ld^2(\Omega;\Hf^{\otimes r})$.
For a random variable $Z\in\Rc$, say $Z=g(\int_{\R^d} G\zeta_1,\ldots,\int_{\R^d} G\zeta_n)$, we define its \emph{Malliavin derivative} $DZ\in \Ld^2(\Omega;\Hf)$ as
\begin{align}\label{eq:D-expl}
D_xZ=\sum_{i=1}^n\zeta_i(x)\,(\nabla_i g)\big({\textstyle\int_{\R^d} G\zeta_1,\ldots,\int_{\R^d} G\zeta_n}\big).
\end{align}
For elements $X\in\Rc(\Hf^{\otimes r})$ with $r\ge1$, say $Z=\sum_{i=1}^n\psi_iZ_i$, the Malliavin derivative $DZ\in\Ld^2(\Omega;\Hf^{\otimes(r+1)})$ is then given by $DZ=\sum_{i=1}^n\psi_i\otimes DZ_i$.
For $j\ge1$, we iteratively define the $j$-th Malliavin derivative $D^j:\Rc(\Hf^{\otimes r})\to\Ld^2(\Omega;\Hf^{\otimes (r+j)})$ for all $r\ge0$.
Next, for $r,m\ge0$, we set
\begin{gather*}
\langle Y,Z\rangle_{\Dm^{m,2}(\Hf^{\otimes r})}:=\expec{\langle Y,Z\rangle_{\Hf^{\otimes r}}}+\sum_{j=1}^m\expec{\langle D^jY,D^jZ\rangle_{\Hf^{\otimes(r+j)}}},
\end{gather*}
we define the {\it Malliavin-Sobolev space} $\Dm^{m,2}(\Hf^{\otimes r})$ as the closure of $\Rc(\Hf^{\otimes r})$ for the corresponding norm, and we extend the Malliavin derivatives $D^j$ by density to these spaces.

Based on these definitions, we state the following proposition, which collects various classical results from Malliavin calculus. Item~(i) is standard (e.g.~\cite{Houdre-PerezAbreu-95}), as well as item~(ii) (e.g.~\cite{Bakry-94} and references therein).
Item~(iii) is a consequence of Stein's method: it was first obtained in the discrete setting by Chatterjee~\cite{Chat08}, while the present Malliavin analogue is due to~\cite{NP-08,NPR-09}.
A short proof can be found e.g.\@ in~\cite[Appendix~A]{DO1}.

\begin{prop}[\cite{Houdre-PerezAbreu-95,Bakry-94,Chat08,NP-08,NPR-09}]\label{prop:Mall}$ $
\begin{enumerate}[(i)]
\item \emph{First-order Poincaré inequality:} For all $Z\in \Ld^2(\Omega)$,
\begin{align*}
\var{Z}\le\expec{\|DZ\|_{\Hf}^2}.
\end{align*}
\item \emph{Logarithmic Sobolev inequality:} For all $Z\in\Ld^2(\Omega)$,
\[\ent{Z^2}:=\expec{Z^2\log\frac{Z^2}{\expec{Z^2}}}\le2\,\expec{\|DZ\|_\Hf^2}.\]
\item \emph{Second-order Poincaré inequality:} For all $Z\in \Ld^2(\Omega)$ with $\expec{Z}=0$ and $\var{Z}=1$,
\[\quad\dW(Z,\Nc)\vee\dTV{Z}{\Nc}\,\le\,3\,\expec{\|D^2Z\|_{\op}^4}^\frac14\expec{\|DZ\|_\Hf^4}^\frac14,\]
where $\dW(\cdot,\Nc)$ and $\dTV\cdot\Nc$ denote the $1$-Wasserstein and the total variation distances to a standard Gaussian law and where the norm of $D^2Z$ is defined by
\begin{equation}\label{eq:def-op}
\|D^2Z\|_{\op}:=\sup_{\zeta,\zeta'\in\Hf\atop\|\zeta\|_\Hf=\|\zeta'\|_\Hf=1}\langle D^2Z,\zeta\otimes\zeta'\rangle_{\Hf^{\otimes2}}.
\qedhere
\end{equation}
\end{enumerate}
\end{prop}

\subsection{Link with functional derivative}
{\color{black}While Malliavin calculus is developed for all $\sigma(G)$-measurable random variables, we now focus on the subspace $\Ld^2_A(\Omega)$ of $\sigma(A)$-measurable random variables in $\Ld^2(\Omega)$. While Malliavin derivative is defined on the regular subspace~$\Rc$ (cf.~\eqref{eq:def-Rc}) and extended by density, we consider the corresponding regular subspace,
\[\Rc_A:=\Big\{g\big({\textstyle\int_{\R^d} A\zeta_1,\ldots,\int_{\R^d} A\zeta_n})\big)\,:\,n\in\N,\,g\in C_c^\infty((\R^d)^n),\,\zeta_1,\ldots,\zeta_n\in C_c^\infty(\R^d)\Big\},\]
which is dense in $\Ld^2_A(\Omega)$.
A simple computation then relates the Malliavin derivative~\eqref{eq:D-expl} to the functional derivative~\eqref{eq:def-frech} commonly considered e.g.\@ in stochastic homogenization.
%the following comparison between the functional derivative~\eqref{eq:def-frech} and the Malliavin derivative~\eqref{eq:D-expl}.

\begin{lem}\label{lem:derivatives}
For all $\sigma(A)$-measurable random variable $Z(A)\in\Rc_A$, there holds
\begin{equation*}
(DZ(\Aa))_j=(\nabla_{j}h)(G)\,\tfrac{\partial Z(\Aa)}{\partial\Aa},
\end{equation*}
and similarly
\begin{align*}
&(D^2_{xy}Z(\Aa))_{jl}\,=\,\delta(x-y)\,(\nabla^2_{jl}h)(G(x))\,\tfrac{\partial Z(\Aa)}{\partial\Aa}(x)\nonumber\\
&\hspace{6cm}+(\nabla_{l}h)(G(y))\,(\nabla_{j}h)(G(x))\,\tfrac{\partial^2Z(\Aa)}{\partial\Aa^2}(x,y).\qedhere
\end{align*}
\end{lem}

\begin{proof}
Let $Z(A)\in\Rc_A$, say $Z(A)=g(\int_{\R^d}A\zeta_1,\ldots,\int_{\R^d}A\zeta_n)$.
Approximating $A:=h(G)$ by $A_\e:=h(\chi_\e\ast G)$ with $\chi_\e=\e^{-d}\chi(\frac1\e\cdot)$, $\chi\in C^\infty_c(\R^d)$, and $\int_{\R^d}\chi=1$, and applying the definition~\eqref{eq:D-expl} of the Malliavin derivative, we easily compute 
\[(D_xZ(A))_j=(\nabla_jh)(G(x))\sum_{i=1}^n\zeta_i(x)\,(\nabla_ig)\big({\textstyle\int_{\R^d}A\zeta_1,\ldots,\int_{\R^d}A\zeta_n}\big).\]
Recognizing the definition of the functional derivative~\eqref{eq:def-frech} in form of
\[\tfrac{\partial Z(A)}{\partial A}(x)=\sum_{i=1}^n\zeta_i(x)\,(\nabla_ig)\big({\textstyle\int_{\R^d}A\zeta_1,\ldots,\int_{\R^d}A\zeta_n}\big),\]
the claim follows for the first Malliavin derivative. The result for the second derivative is obtained similarly.
\end{proof}

Next, while the size of the Malliavin derivative is naturally measured in the norm of $\Hf$, we proceed to a radial change of variables to transform this Hilbert norm into a multiscale weighted norm as in the proof of~\cite[Theorem~3.1]{DG2}.}
\begin{lem}\label{lem:bound-weight}
Assume that $|\Cc(x)|\le c(|x|)$ for some decreasing function $c:\R_+\to\R_+$.
Then, for all $H\in C^\infty_c(\R^d)^k$,
\[\|H\|_{\Hf}^2\,\lesssim\, \int_0^\infty \int_{\R^d}\bigg(\int_{B_{\ell+1}(x)}|H|\bigg)^2dx\,(\ell+1)^{-d}\,(-c'(\ell))\,d\ell.\]
In particular, if $\Cc$ is integrable in the sense of $\int_0^\infty(\ell+1)^d\,(-c'(\ell))\,d\ell\lesssim1$, we find
\[\|H\|_{\Hf}^2\,\lesssim\,\int_{\R^d}\bigg(\int_{B(x)}|H|\bigg)^2dx.\qedhere\]
\end{lem}

\begin{proof}
Using $|\Cc(x)|\le c(|x|)$, taking local averages, passing to spherical coordinates, and integrating by parts, we find
\begin{eqnarray*}
\|H\|_{\Hf}^2
&\lesssim & \int_{\R^d}|H(x)|\int_{\Sp^{d-1}}\int_0^\infty\Big(\fint_{B(x+\ell u)} |H|\Big)\ell^{d-1}\,c(\ell)\,d\ell\,d\sigma(u)\,dx\\
&=& \int_{\R^d}|H(x)|\int_{\Sp^{d-1}}\int_0^\infty\int_0^\ell
\Big(\fint_{B(x+s u)}|H|\Big) s^{d-1}ds\,(-c'(\ell))\,d\ell\,d\sigma(u)\,dx\\
&\lesssim& \int_{\R^d} |H(x)| \int_0^\infty\bigg(\int_{B_{\ell+1}(x)}|H|\bigg)\,(-c'(\ell))\,d\ell\,dx.
\end{eqnarray*}
Taking local spatial averages, this turns into
\begin{eqnarray*}
{\|H\|_{\Hf}^2}
&\lesssim& \int_0^\infty \int_{\R^d}\int_{B_{\ell+1}}|H(x+y)|\bigg(\int_{B_{\ell+1}(x+y)}|H|\bigg)dydx\,(\ell+1)^{-d}\,(-c'(\ell))\, d\ell\\
&\le& \int_0^\infty \int_{\R^d}\bigg(\int_{B_{2(\ell+1)}(x)}|H|\bigg)^2dx\,(\ell+1)^{-d}\,(-c'(\ell))\,d\ell.
\end{eqnarray*}
Covering balls of size $2(\ell+1)$ with balls of size~$\ell+1$, this yields the first claim. The second claim follows from Jensen's inequality.
\end{proof}

\subsection{Multiscale first- and second-order Poincaré inequalities}
Combining Proposition~\ref{prop:Mall}(i)--(ii) with Lemmas~\ref{lem:derivatives} and~\ref{lem:bound-weight}, we deduce
\begin{eqnarray*}
\var{Z(\Aa)}&\le&\|\nabla h\|_{\Ld^\infty}^2\expec{\int_0^\infty \int_{\R^d}|\parfct{A,B_{\ell+1}(x)}Z(\Aa)|^2\,dx\,(\ell+1)^{-d}\,(-c'(\ell))\,d\ell},\\
\ent{Z(\Aa)^2}&\le&2\|\nabla h\|_{\Ld^\infty}^2\expec{\int_0^\infty \int_{\R^d}|\parfct{A,B_{\ell+1}(x)}Z(\Aa)|^2\,dx\,(\ell+1)^{-d}\,(-c'(\ell))\,d\ell}.
\end{eqnarray*}
This constitutes an alternative proof of~\cite[Theorem~3.1]{DG2}.
We turn to the case of second-order inequalities and establish Theorem~\ref{th:fctinequ}.

\begin{proof}[Proof of Theorem~\ref{th:fctinequ}]
Setting $\sigma^2:=\var{Z(\Aa)}$, we apply Proposition~\ref{prop:Mall}(iii) in the form
\begin{multline*}
\dW\big(\sigma^{-1}(Z(\Aa)-\expec{Z(\Aa)})\,,\,\Nc\big)+\dTV{ \sigma^{-1}(Z(\Aa)-\expec{Z(\Aa)})\,}{\,\Nc}\\
\,\lesssim\,\sigma^{-2}\,\expec{\|D^2Z(\Aa)\|_{\op}^4}^\frac14\expec{\|DZ(\Aa)\|_\Hf^4}^\frac14,
\end{multline*}
and it remains to estimate the two RHS factors in terms of functional derivatives.
By Lemmas~\ref{lem:derivatives} and~\ref{lem:bound-weight} with $|\Cc(x)|\lesssim(1+|x|)^{-\beta}$, we find
\[\expec{\|DZ(\Aa)\|_\Hf^4}^\frac14\,\le\,\|\nabla h(G)\|_{\Ld^\infty}\expec{\3\parfct{A}Z(A)\3_\beta^4}^\frac14,\]
with $\3\cdot\3_\beta$ defined as in the statement. We turn to the second Malliavin derivative.
Setting $\Cc=\Cc_0\ast \Cc_0$ and noting that $\|\zeta\|_{\Hf}=\|\Cc_0\ast\zeta\|_{\Ld^2}$, we find
\begin{align*}
\|D^2Z(A)\|_{\op}
=\sup_{\|\xi\|_{\Ld^2}=\|\xi'\|_{\Ld^2}=1}\int_{\R^d}\int_{\R^d}D_{xy}^2Z(A):\Cc_0\ast\xi(x)\,\otimes\,\Cc_0\ast\xi'(y)\,dxdy,
\end{align*}
hence, by Lemma~\ref{lem:derivatives},
\begin{multline}\label{eq:decomp-D2X}
\|D^2Z(A)\|_{\op}
\le\|\nabla^2h\|_{\Ld^\infty}\sup_{\|\xi\|_{\Ld^2}=\|\xi'\|_{\Ld^2}=1}\int_{\R^d}\Big|\frac{\partial Z(A)}{\partial A}\Big|\,|\Cc_0\ast\xi|\,|\Cc_0\ast\xi'|\\
+\|\nabla h\|_{\Ld^\infty}^2\sup_{\|\xi\|_{\Ld^2}=\|\xi'\|_{\Ld^2}=1}\int_{\R^d}\int_{\R^d}\Big|\frac{\partial^2Z(\Aa)}{\partial\Aa^2}(x,y)\Big|\,|\Cc_0\ast\xi(x)|\,|\Cc_0\ast\xi'(y)|\,dxdy.
\end{multline}
We start with the first RHS term. For all $p\ge1$, we may estimate
\begin{multline*}
\int_{\R^d}\Big|\frac{\partial Z(A)}{\partial A}\Big|\,|\Cc_0\ast\xi|\,|\Cc_0\ast\xi'|\,\lesssim\,\Big(\int_{\R^d}|\parfct{A,B(z)}Z(A)|^p\,dz\Big)^\frac1p\\
\times\Big(\int_{\R^d}\sup_{B(z)}|\Cc_0\ast\xi|^{\frac{2p}{p-1}}\,dz\Big)^\frac{p-1}{2p}\Big(\int_{\R^d}\sup_{B(z)}|\Cc_0\ast\xi'|^{\frac{2p}{p-1}}\,dz\Big)^\frac{p-1}{2p}.
\end{multline*}
Recall that we assume $|\Cc(x)|\lesssim(1+|x|)^{-\beta}$ and $|\Cc_0(x)|\lesssim(1+|x|)^{-(d+\beta)/2}$.
For $\beta>d$, we find $\int_{\R^d}\sup_{B(z)}|\Cc_0\ast\xi|^2dz\lesssim\|\xi\|_{\Ld^2}^2$, hence for all $p\ge1$,
\begin{align*}
\int_{\R^d}\Big|\frac{\partial Z(A)}{\partial A}\Big|\,|\Cc_0\ast\xi|\,|\Cc_0\ast\xi'|\,\lesssim\,\|\xi\|_{\Ld^2}\|\xi'\|_{\Ld^2}\Big(\int_{\R^d}|\parfct{A,B(z)}Z(A)|^p\,dz\Big)^\frac1p.
\end{align*}
For $\beta<d$, the Hardy-Littlewood-Sobolev inequality rather yields $\int_{\R^d}\sup_{B(z)}|\Cc_0\ast\xi|^{2d/\beta}dz\lesssim\|\xi\|_{\Ld^2}^{2d/\beta}$, hence the above estimate then only holds for all $1\le p\le\frac{d}{d-\beta}$.
It remains to analyze the second RHS term in~\eqref{eq:decomp-D2X}.
Computing the supremum and setting $\tilde\Cc:=|\Cc_0|\ast|\Cc_0|$, we find
\begin{multline*}
\sup_{\|\xi\|_{\Ld^2}=\|\xi'\|_{\Ld^2}=1}\int_{\R^d}\int_{\R^d}\Big|\frac{\partial^2Z(\Aa)}{\partial\Aa^2}(x,y)\Big|\,|\Cc_0\ast\xi(x)|\,|\Cc_0\ast\xi'(y)|\,dxdy\\
\,\lesssim\,\bigg(\int_{(\R^d)^2}\int_{(\R^d)^2}\Big|\frac{\partial^2Z(\Aa)}{\partial\Aa^2}(x,y)\Big|\Big|\frac{\partial^2Z(\Aa)}{\partial\Aa^2}(x',y')\Big|\tilde\Cc(x-x')\tilde\Cc(y-y')\,dxdydx'dy'\bigg)^\frac12.
\end{multline*}
Using $|\tilde\Cc(x)|\lesssim(1+|x|)^{-\beta}$ and proceeding to a radial change of variables as in the proof of Lemma~\ref{lem:bound-weight}, the conclusion follows.
\end{proof}
}

\section*{Acknowledgements}
We warmly thank Ivan Nourdin for explanations on second order Poincar\'e inequalities.
We acknowledge financial support from the European Research Council under
the European Community's Seventh Framework Programme (FP7/2014-2019 Grant Agreement
QUANTHOM 335410). 

\addtocontents{toc}{\protect\setcounter{tocdepth}{1}}
\bigskip
\bibliographystyle{plain}
%\bibliography{biblio}

\begin{thebibliography}{10}

\bibitem{Baddeley-80}
A.~Baddeley.
\newblock A limit theorem for statistics of spatial data.
\newblock {\em Adv. in Appl. Probab.}, 12(2):447--461, 1980.

\bibitem{Bakry-94}
D.~Bakry.
\newblock L'hypercontractivit\'e et son utilisation en th\'eorie des
  semigroupes.
\newblock In {\em Lectures on probability theory ({S}aint-{F}lour, 1992)},
  volume 1581 of {\em Lecture Notes in Math.}, pages 1--114. Springer, Berlin,
  1994.

\bibitem{Chat08}
S.~Chatterjee.
\newblock A new method of normal approximation.
\newblock {\em Ann. Probab.}, 36(4):1584--1610, 2008.

\bibitem{Chat-09}
S.~Chatterjee.
\newblock Fluctuations of eigenvalues and second order {P}oincar\'e
  inequalities.
\newblock {\em Probab. Theory Related Fields}, 143(1-2):1--40, 2009.

\bibitem{DFG}
M.~Duerinckx, J.~Fischer, and A.~Gloria.
\newblock Scaling limit of the homogenization commutator for {G}aussian
  coefficient fields.
\newblock Preprint.

\bibitem{DG1}
M.~Duerinckx and A.~Gloria.
\newblock {Multiscale functional inequalities in probability: Concentration
  properties}.
\newblock Preprint, arXiv:1711.03148, 2018.

\bibitem{DG2}
M.~Duerinckx and A.~Gloria.
\newblock {Multiscale functional inequalities in probability: Constructive
  approach}.
\newblock Preprint, arXiv:1711.03152, 2018.

\bibitem{DGO1}
M.~Duerinckx, A.~Gloria, and F.~Otto.
\newblock The structure of fluctuations in stochastic homogenization.
\newblock Preprint, arXiv:1602.01717, 2017.

\bibitem{DGO2}
M.~Duerinckx, A.~Gloria, and F.~Otto.
\newblock Robustness of the pathwise structure of fluctuations in stochastic
  homogenization.
\newblock Preprint, arXiv:1807.11781, 2018.

\bibitem{DO1}
M.~Duerinckx and F.~Otto.
\newblock Higher-order pathwise theory of fluctuations in stochastic
  homogenization.
\newblock Preprint, arXiv:1903.02329, 2019.

\bibitem{Efron-Stein-81}
B.~Efron and C.~Stein.
\newblock The jackknife estimate of variance.
\newblock {\em Ann. Statist.}, 9(3):586--596, 1981.

\bibitem{GN}
A.~Gloria and J.~Nolen.
\newblock A quantitative central limit theorem for the effective conductance on
  the discrete torus.
\newblock {\em Comm. Pure Appl. Math.}, 69(12):2304--2348, 2016.

\bibitem{Gu-Bal-12}
Y.~Gu and G.~Bal.
\newblock Random homogenization and convergence to integrals with respect to
  the {R}osenblatt process.
\newblock {\em J. Differential Equations}, 253(4):1069--1087, 2012.

\bibitem{GuM}
Y.~Gu and J.-C. Mourrat.
\newblock Scaling limit of fluctuations in stochastic homogenization.
\newblock {\em Multiscale Model. Simul.}, 14(1):452--481, 2016.

\bibitem{Heinrich-05}
L.~Heinrich.
\newblock Large deviations of the empirical volume fraction for stationary
  {P}oisson grain models.
\newblock {\em Ann. Appl. Probab.}, 15(1A):392--420, 2005.

\bibitem{Houdre-PerezAbreu-95}
C.~Houdr\'e and V.~P\'erez-Abreu.
\newblock Covariance identities and inequalities for functionals on {W}iener
  and {P}oisson spaces.
\newblock {\em Ann. Probab.}, 23(1):400--419, 1995.

\bibitem{HLS-16}
D.~Hug, G.~Last, and M.~Schulte.
\newblock Second-order properties and central limit theorems for geometric
  functionals of {B}oolean models.
\newblock {\em Ann. Appl. Probab.}, 26(1):73--135, 2016.

\bibitem{Lachieze-Peccati-16}
R.~Lachi\`eze-Rey and G.~Peccati.
\newblock New {B}erry-{E}sseen bounds for functionals of binomial point
  processes.
\newblock {\em Ann. Appl. Probab.}, 27(4):1992--2031, 2017.

\bibitem{LPS-16}
G.~Last, G.~Peccati, and M.~Schulte.
\newblock Normal approximation on {P}oisson spaces: {M}ehler's formula, second
  order {P}oincar\'e inequalities and stabilization.
\newblock {\em Probab. Theory Related Fields}, 165(3-4):667--723, 2016.

\bibitem{LNZH-17}
A.~Lechiheb, I.~Nourdin, G.~Zheng, and E.~Haouala.
\newblock Convergence of random oscillatory integrals in the presence of
  long-range dependence and application to homogenization.
\newblock {\em Probab. Math. Statist.}, 38(2):271--286, 2018.

\bibitem{Lee-97}
S.~Lee.
\newblock The central limit theorem for {E}uclidean minimal spanning trees.
  {I}.
\newblock {\em Ann. Appl. Probab.}, 7(4):996--1020, 1997.

\bibitem{Lee-99}
S.~Lee.
\newblock The central limit theorem for {E}uclidean minimal spanning trees.
  {II}.
\newblock {\em Adv. in Appl. Probab.}, 31(4):969--984, 1999.

\bibitem{Malliavin-97}
P.~Malliavin.
\newblock {\em Stochastic analysis}, volume 313 of {\em Grundlehren der
  Mathematischen Wissenschaften [Fundamental Principles of Mathematical
  Sciences]}.
\newblock Springer-Verlag, Berlin, 1997.

\bibitem{Mase-82}
S.~Mase.
\newblock Asymptotic properties of stereological estimators of volume fraction
  for stationary random sets.
\newblock {\em J. Appl. Probab.}, 19(1):111--126, 1982.

\bibitem{MN}
J.-C. Mourrat and J.~Nolen.
\newblock Scaling limit of the corrector in stochastic homogenization.
\newblock {\em Ann. Appl. Probab.}, 27(2):944--959, 2017.

\bibitem{Nolen-13}
J.~Nolen.
\newblock Normal approximation for a random elliptic equation.
\newblock {\em Probab. Theory Related Fields}, 159(3):661--700, 2014.

\bibitem{NP-08}
I.~Nourdin and G.~Peccati.
\newblock Stein's method on {W}iener chaos.
\newblock {\em Probab. Theory Related Fields}, 145(1-2):75--118, 2009.

\bibitem{Nourdin-Peccati-12}
I.~Nourdin and G.~Peccati.
\newblock {\em Normal approximations with {M}alliavin calculus}, volume 192 of
  {\em Cambridge Tracts in Mathematics}.
\newblock Cambridge University Press, Cambridge, 2012.
\newblock From Stein's method to universality.

\bibitem{NPR-09}
I.~Nourdin, G.~Peccati, and G.~Reinert.
\newblock Second order {P}oincar\'e inequalities and {CLT}s on {W}iener space.
\newblock {\em J. Funct. Anal.}, 257(2):593--609, 2009.

\bibitem{Penrose-01}
M.~D. Penrose.
\newblock Random parking, sequential adsorption, and the jamming limit.
\newblock {\em Comm. Math. Phys.}, 218(1):153--176, 2001.

\bibitem{Penrose-05}
M.~D. Penrose.
\newblock Multivariate spatial central limit theorems with applications to
  percolation and spatial graphs.
\newblock {\em Ann. Probab.}, 33(5):1945--1991, 2005.

\bibitem{Penrose-Yukich-02}
M.~D. Penrose and J.~E. Yukich.
\newblock Limit theory for random sequential packing and deposition.
\newblock {\em Ann. Appl. Probab.}, 12(1):272--301, 2002.

\bibitem{PY-05}
M.~D. Penrose and J.~E. Yukich.
\newblock Normal approximation in geometric probability.
\newblock In {\em Stein's method and applications}, volume~5 of {\em Lect.
  Notes Ser. Inst. Math. Sci. Natl. Univ. Singap.}, pages 37--58. Singapore
  Univ. Press, Singapore, 2005.

\bibitem{Schreiber-Penrose-Yukich-07}
T.~Schreiber, M.~D. Penrose, and J.~E. Yukich.
\newblock Gaussian limits for multidimensional random sequential packing at
  saturation.
\newblock {\em Comm. Math. Phys.}, 272(1):167--183, 2007.

\bibitem{Steele-86}
J.M. Steele.
\newblock An {E}fron-{S}tein inequality for nonsymmetric statistics.
\newblock {\em Ann. Statist.}, 14:753--758, 1986.

\bibitem{Taqqu}
M.~S. Taqqu.
\newblock Convergence of integrated processes of arbitrary {H}ermite rank.
\newblock {\em Z. Wahrsch. Verw. Gebiete}, 50(1):53--83, 1979.

\bibitem{Torquato-02}
S.~Torquato.
\newblock {\em Random heterogeneous materials. Microstructure and macroscopic
  properties}, volume~16 of {\em Interdisciplinary Applied Mathematics}.
\newblock Springer-Verlag, New York, 2002.

\bibitem{TUS-06}
S.~Torquato, O.U. Uche, and F.H. Stillinger.
\newblock Random sequential addition of hard spheres in high euclidean
  dimensions.
\newblock {\em Phys. Rev. E}, 74:061308, 2006.

\bibitem{Vidotto}
A.~{Vidotto}.
\newblock {An Improved Second Order Poincaré Inequality for Functionals of
  Gaussian Fields}.
\newblock Preprint, arXiv:1706.06985, 2017.

\end{thebibliography}

\def\cprime{$'$} \def\cprime{$'$}

\end{document}